\documentclass[a4paper,11pt,draft]{article}
\usepackage[papersize={210truemm,297truemm},margin=3.0truecm,nohead]{geometry}
\usepackage[utf8]{inputenc}
\usepackage{amsmath}						
\usepackage{amsfonts}
\usepackage{latexsym}
\usepackage{amssymb}
\usepackage{bm}
\usepackage{ulem}
\usepackage{cases}

\allowdisplaybreaks[3]

\usepackage{amsthm}
\usepackage{ascmac}
\usepackage{mathtools}

\theoremstyle{definition}
\newtheorem{theorem}{Theorem}[section]
\newtheorem{proposition}[theorem]{Proposition}
\newtheorem{lemma}[theorem]{Lemma}
\newtheorem{corollary}[theorem]{Corollary}
\newtheorem{remark}[theorem]{Remark}
\newtheorem{definition}[theorem]{Definition}
\newtheorem{example}[theorem]{Example}
\newtheorem*{theorem*}{Theorem}
\newtheorem*{definition*}{Definition}
\newtheorem*{lemma*}{Lemma}
\newtheorem*{proposition*}{Proposition}
\newtheorem*{corollary*}{Corollary}
\newtheorem*{example*}{Example}
\newtheorem*{remark*}{Remark}
\newtheorem*{claim*}{Claim}

\numberwithin{equation}{section}

\def\diam{\mathrm{diam}}

\def\natural{\mathbb{N}}
\def\naturalz{\mathbb{N}_{0}}

\def\real{\mathbb{R}}
\def\complex{\mathbb{C}}
\def\tree{\mathbb{T}}

\title{Understanding the limit sets generated by general iterated function systems on unbounded spaces}
\author{Kanji Inui\footnote{corresponding author}\\
Center for Mathematics, School of Fundamental Science and Technology, \\
Faculty of Science and Technology, Keio University\\
3-14-1, Hiyoshi, Kohoku-ku, Yokohama, 223-8522, JAPAN\\
E-mail: k\_inui@keio.jp\\
}
\begin{document}							
\maketitle

\begin{abstract}
In this paper, we reformulate the definition of the iterated function systems (denoted by general IFSs in this paper) and show the existence and uniqueness (in some sense) of the limit sets generated by the general IFSs, to unify the definitions of the limit sets introduced before. 
Note that the general IFSs are defined on (possibly unbounded) complete metric spaces and we instead assume a ``natural" condition of general IFSs to show the main result. 
To obtain the main result, we apply techniques in the Banach fixed point theorem to the general IFSs with the ``natural" condition. 
Besides, we consider an example of general IFSs. 
\end{abstract}
\setcounter{subsection}{0}
\section{Introduction}
Since many researchers recognized the importance of fractals and have studied iterated function systems (for short, IFSs) in the 1970s, IFS is one of the mathematically powerful tools to construct fractals (often called limit sets) and indeed there are many mathematical papers on the limit sets generated by IFSs. 
In particular, the limit sets generated by iterated function systems with finitely many mappings (henceforth, denoted by autonomous IFSs) have been well-studied, and there exist many results not only on some properties of the limit sets but also on dimensions and measures of the limit sets (\cite{H}, \cite{BG}, \cite{Sch}, \cite{F}, \cite{B}, \cite{Ki}, \cite{KTMV}). 
Note that the limit sets generated by autonomous IFSs automatically have a nice property (self-similarity), which deduces rich results on the limit sets. 

On the other hand, there exist results on the limit sets which are not generated by autonomous IFSs (henceforth, denoted by generalized IFSs) and there are at least three generic lines of studies of the limit sets generated by the generalized IFSs. 
The study in the first line shows the estimates on the dimensions of the limit sets under the assumption of the existence of the limit sets (Moran set, see \cite{Mo}, \cite{HRWW}, \cite{HZ}, \cite{LLMX}, \cite{GM}) and these papers also include examples and applications of the results. 
The study in the second line ensures the existence of the limit sets generated by the non-autonomous (1-variable) IFSs (and some generalizations, see \cite{ReU}, \cite{A}, \cite{LDV}, \cite{DLM}, \cite{Mas}, \cite{N}) and papers \cite{ReU}, \cite{A} and \cite{N} also gives the theorems on the estimation of the dimensions (and measures) of the limit sets. 
The study in the third line shows the existence of the limit sets generated by generalized IFSs (V-variable fractals) and the theorems on the estimation of the Hausdorff dimension of the limit sets by using probabilistic techniques (\cite{BHS1}, \cite{BHS2}, \cite{Sce}). 

These studies indicate there is room to analyze not only the limit sets generated by autonomous IFSs but also the ones generated by generalized IFSs. 
But, these theories have proposed different definitions and assumptions for generalized IFSs, which deduce the results on the estimation of dimensions (and measures) of the limit set generated by each generalized IFS. 
In addition, it is worth mentioning that the  above papers consider the generalized IFSs defined on bounded sets (in some sense) or compact sets. 
Indeed, the papers in the first line consider generalized IFSs under the assumptions which allow us to restrict the domains of the IFSs to a (common) bounded set. 
The papers in the second line consider the non-autonomous (1-variable) IFSs on compact sets or bounded sets (in some sense) to obtain some results (in particular, the existence of the limit sets). 
The papers in the third line consider generalized IFSs under the assumption which implicitly allows us to restrict the domains of the IFSs to a (common) bounded set (see, Remark \ref{limit_set_boundedness_remark}). 

To address the issue, it is important to understand the connection of the definitions of the limit sets generated by the generalized IFSs, and it is natural to reformulate generalized IFSs (henceforth, these reformulated IFSs are denoted by general IFSs) to unify the definitions of the limit sets introduced before. 
Therefore, the aim of this paper is to present the reformulation of the definition of generalized IFSs to show the existence and uniqueness (in some sense) of the limit sets generated by the general IFSs and to unify the definitions of the limit sets introduced before. 

More precisely, we first introduce IFSs which consist of a family of uniformly contractive mappings and a set of all ``infinite words" (called a tree, see Definition \ref{def_of_tree}). 
We next define the projection map for general IFSs by using the ``compatible" sequence for a non-autonomous (recursive) iteration generated by ``infinite words" and the family of the contractive mappings of the general IFS, under the ``natural" condition (see Lemma \ref{projection_map_conti_lemma} and Definition \ref{definition_of_projection_map}).  
Note that the notion of the projection map is already introduced in the second and third line (called the address map in the third line), and we show that the projection map in this paper indeed coincides with the projection map introduced in the second and third line (see, Proposition \ref{justification_projection_map} and Proposition \ref{projection_map_connection_remark}). 
Then, we now construct the family of limit sets for general IFSs by using the projection map for the general IFS and show the uniqueness (in some sense) of the family, under the ``natural" condition (see, Theorem \ref{main} (the main result)). 
Note that the idea of the family of the limit sets is already introduced in the first line (called the basic sets with the Moran structure in the first line), and also note that the definition of the family of limit sets is derived from the definition of limit sets in the second line. 
In addition, we show that the family of the limit sets in this paper is compatible with the definition of basic sets with the Moran structure, and each limit set in the family is expressed as the limit point of the iterations in sense of the Hausdorff distance. 
It follows that the family of the limit sets in this paper is a generalization of the limit sets in the third line. 
We remark that, under the ``natural" condition, we do not assume that general IFSs are defined on bounded sets (in some sense) or compact sets. 

Moreover, to indicate the importance of the general IFSs, we give an example of general IFSs which has a connection to the theory of continued fractions, and we obtain the result on the dimension of the limit set generated by the IFS in this example (see, Proposition  \ref{continued_fraction__elementray_lemma}). 
Note that, while the theory of continued fractions is often discussed in the theory of autonomous IFSs (see \cite{MU}, \cite{MU2}), it is not often discussed in the theory of generalized IFSs (you can find a recent paper \cite{NT} in the setting for non-autonomous IFSs). 
In addition, while we already obtain the existence of the limit set generated by the IFS in the example by applying results in the third line, this example is not much paid attention to the limit set since it does not satisfy the central condition (the $V$-variability). 

Before we present the strategy of the main result in this paper, we recall the Hutchinson technique which is a technique to construct the limit sets generated by autonomous IFSs (in detail, see \cite{Ki}). 
Indeed, We first consider a complete metric space $X$ (which is possibly unbounded) and the set of all non-empty compact subsets of $X$ with the Hausdorff distance (denoted by $(\mathcal{K}(X), d_{H})$). 
Note that $(\mathcal{K}(X), d_{H})$ is complete since $X$ is complete. 
In the Hutchinson technique, for an autonomous IFS, we introduce an operator on $(\mathcal{K}(X), d_{H})$ associated with the autonomous IFS (called the Barnsley operator) and show the contractivity of the Barnsley operator on $(\mathcal{K}(X), d_{H})$. 
By the Banach fixed point theorem, we deduce that there exists the unique fixed point $K \in \mathcal{K}(X)$ (the unique non-empty compact subset) of the Barnsley operator on $(\mathcal{K}(X), d_{H})$ and the unique fixed point $K$ is called the limit set for the autonomous IFS (or called the self-similar set in this context). 
In addition, there is a connection between the limit sets generated by the autonomous IFS and a set of all infinite sequences of symbols (called the symbolic space). 
Indeed, recall that each point of the limit sets is expressed as some (recursive) iteration generated by the autonomous IFS. 
Since each iteration is expressed as the infinite sequence of the symbols, there is a ``nice" map on the symbolic space such that the image of the ``nice" map equals the limit set, and the ``nice" map is called the code map. 
Note that the projection map is a generalization of the code map. 
Also, note that neither compactness nor boundedness of $X$ is not assumed in the above arguments. 
Later we find that, in the theory of generalized IFSs, the compactness or boundedness (in some sense) of $X$ is a sufficient condition to obtain the existence and uniqueness of the limit sets (see, Theorem \ref{main} (the main result)). 

Now, we give the strategy to obtain the main results (Theorem \ref{main}). 
To obtain the projection map for general IFSs, we first consider non-autonomous (recursive) iterations generated by a sequence of uniformly contractive mappings on a complete metric space, and we recall results on the existence and uniqueness of the ``compatible" sequence for the non-autonomous (recursive) iterations and some properties of the ``compatible" sequence. 
Note that we need the ``natural" condition to show the above results by using the techniques in the Banach fixed point theorem (indeed there is a counterexample, see Example \ref{1dimex}). 
By the above argument, we next define the projection map for a general IFS and discuss some properties of the projection map by applying the above results to the non-autonomous (recursive) iterations generated by the general IFS on a complete metric space. 
In particular, we show the continuity of the projection map for general IFSs. 
Now, we finally construct the family of the limit sets for a general IFS by the continuity of the projection map. 
Then we obtain the uniqueness (in some sense) of the family of the limit sets and show that each limit set in the family is expressed as the limit point of the iterations in sense of the Hausdorff distance (the Theorem \ref{main} (the main result)) by the Hutchinson technique. 
To this end, there are two points to show the main theorem. 
Indeed, by using some properties of the Hausdorff distance and the assumption in the definition of the general IFS, we first show that the convergence of limit sets in sense of the Hausdorff distance with the initial compact set $\{x\} \in \mathcal{K}(X)$ ($x \in X$) and next show that the convergence of limit sets in sense of the Hausdorff distance does not depend on the initial compact set. 

The rest of the paper is organized as follows. 
In Section 2, we recall some basic properties of non-autonomous (right) iterations generated by a sequence of contractive mappings on complete metric spaces. 
In addition, we also present some examples of non-autonomous iterations in this section. 
In Section 3, we give the definitions of general IFSs and the projection map for general IFSs. 
Indeed, we first introduce the notion of trees and discuss some properties of the trees. 
We next introduce the definition of general IFSs on complete metric spaces and the projection map for general IFSs. 
We also discuss the properties of the projection map for general IFSs in this section. 
In Section 4, we finally construct the family of limit sets generated by general IFSs and show the uniqueness of the family (in some sense) and properties of the family (the main result). 
In Section 5, we give an example of general IFSs and discuss its properties. 
%
%
%
%
%
%
%
%
%
%
%
%
%
%
%
%
%
%
%
%
%
%
%
%
%
%
%
\section{Preliminaries}
In this section, we recall non-autonomous (right) iterations of contractive mappings on a complete metric space to consider general IFSs. 
In Subsection 2.1, we first present the existence and uniqueness (in some sense) of the recursively compatible sequence (see, Definition \ref{definition_of_nonautonomous_iterations}) under a ``natural" condition. 
Note that the recursively compatible sequence is a generalization of fixed points in the theory of dynamical systems. 
In Subsection 2.2, we present some examples of non-autonomous iterations of contractive mappings. 

Now, we first introduce the definition of sequence of contractive mapping with uniform contraction constant. 
Henceforth, $\natural$ is the set of positive integers and $\naturalz$ is the set of non-negative integers. 

\begin{definition}\label{definition_of_nonautonomous_iterations}
    We say that $f_{j} \colon X \to X \ (j \in \natural)$ be a sequence of contractive mappings on a complete metric space $(X, \rho)$ with an uniform contraction constant $c \in (0, 1)$ if 
    \begin{equation*} \label{basicnonautonomous}
        \rho(f_{j}(x), f_{j}(y)) \leq c \ \rho(x, y)
    \end{equation*}
     for all $j \in \natural$ and $x, y \in X$. 
    Let $\{ f_{j} \}_{j \in \natural}$ be a sequence of contractive mappings on a complete metric space $(X, \rho)$ with an uniform contraction constant $c \in (0, 1)$. 
    We say that $\{x_{m}\}_{m \in \natural} \subset X$ is a recursively compatible sequence for $\{ f_{j} \}_{j \in \natural}$ if
    $f_{m}(x_{m+1}) = x_{m}$ for each $m \in \natural$. 
\end{definition}
Note that there exists the unique fixed point $z_{j}$ of $f_{j}$ for each $j \in \natural$ since $X$ is complete. 
We set $Z := \{ z_{j} \in X \ | \ j \in \natural \}$. 

Let $f_{j} \colon X \to X \ (j \in \natural)$ be a sequence of contractive mappings on a complete metric space $(X, \rho)$ with an uniform contraction constant $c \in (0, 1)$. 
Then, for each $m \in \natural$, we call $\{f_{[m, n]} \}_{n \geq m}$ and $\{f_{[m, n)}\}_{n > m}$ a non-autonomous (right) iteration of the contractive mappings on $X$ with an uniform contraction constant $c \in (0, 1)$. 
Here, $f_{[m, n]} \colon X \to X$ and $f_{[m, n)} \colon X \to X$ are defined by
\begin{align}\label{nonautonomousrightiteratedsequence}
    f_{[m, n]} &:= f_{m} \circ \cdots \circ f_{n} \quad (n \geq m) \quad \text{and} \quad 
    f_{[m, n)} := f_{m} \circ \cdots \circ f_{n-1} \quad (n > m). \notag
\end{align}
Also, for each $m \in \naturalz$, the mappings $f_{(m, n]} \colon X \to X$ and $f_{(m, n)} \colon X \to X$ is defined by 
\begin{align*}
    f_{(m, n]} &:= f_{m+1} \circ \cdots \circ f_{n} \quad \text{and} \quad 
    f_{(m, n)} := f_{m+1} \circ \cdots \circ f_{n-1}
\end{align*}
respectively if the relation of $m, n \in \naturalz$ is compatible with compositions of contractive mappings $f_{j} \colon X \to X$ ($j \in \natural$).  

%
%
%
%
%
%
%
%
%
%
%
%
%
%
%
%
%
%
%
%
%
%
%
%
%
%
%
%
%
%
%
%
%
%
%
%
%
%
%
%
%
%
%
%
%
%
\subsection{Basic properties of non-autonomous iterations on complete metric spaces} 
In this subsection, we recall some results on non-autonomous iterations which is a slight generalization of Banach's fixed point theorem. 
Note that we do not assume that $Z$ is bounded if we do not mention it in the statements. 
For the readers, we give a proof of the results. 

Henceforth, we sometimes refer $x \in X$ as a base point of $X$ and $y \in X$ as the starting point of $X$. 
Before we present the main result in this subsection, we give the following lemmas. 

\begin{lemma} \label{Cauchyfiniteness} 
    Let $f_{j} \colon X \to X \ (j \in \natural)$ be a sequence of contractive mappings on a complete metric space $(X, \rho)$ with an uniform contraction constant $c \in (0, 1)$ and $z_{j} \in X \ (j \in \natural)$ be the unique fixed point of $f_{j}$. 
    If there exists $x \in X$ such that 
    $\sum_{j \in \natural} c ^{j} \rho(x, z_{j}) < \infty, $
    then $\sum_{j \in \natural} c ^{j} \rho(x^{\prime}, z_{j}) < \infty$ for each $x^{\prime} \in X$. 
\end{lemma}
\begin{proof}
    Let $x^{\prime} \in X$ and $x \in X$ with $\sum_{j \in \natural} c ^{j} \rho(x, z_{j}) < \infty$. 
    Then, 
    \begin{equation*}
        \sum_{j \in \natural} c ^{j} \rho(x^{\prime}, z_{j}) \leq \rho(x^{\prime}, x) \sum_{j \in \natural} c ^{j} + \sum_{j \in \natural} c ^{j} \rho(x, z_{j}) < \infty 
    \end{equation*}
    since $\sum_{j \in \natural} c ^{j} < \infty$. 
    Therefore, we have proved our lemma. 
\end{proof}
\begin{remark}
    By Lemma \ref{Cauchyfiniteness}, if there exists $x \in X$ such that $\sum_{j \in \natural} c^{j} \rho(x, z_{j}) = \infty$, then $\sum_{j \in \natural} c^{j} \rho(x^{\prime}, z_{i}) = \infty$ for each $x^{\prime} \in X$. 
    Therefore, the property $\sum_{j \in \natural} c ^{j} \rho(x, z_{j}) < \infty$ does not depend on the point $x \in X$ but depend on the iteration $\{f_{j}\}_{j \in \natural}$. 
\end{remark}
\begin{lemma}[Collage theorem, Inverse collage theorem \cite{KTMV}]\label{collage}
    Let $f \colon X \to X $ be a contractive mapping on a complete metric space $(X, \rho)$ with a contraction constant $c \in (0, 1)$. 
    Let $z \in X$ be the unique fixed point of $f$. 
    Then, for each $a \in X$, we have
    \begin{equation*}
        \rho(f(a), a) \leq (1 + c) \ \rho(z, a) \quad \text{and} \quad \rho(z, a) \leq \frac{\rho(f(a), a)}{1 - c}. 
    \end{equation*}
\end{lemma}
We now present a important lemma which is used Sections \ref{general_IFS} and \ref{limit_set}. 
\begin{lemma}\label{e_contractiveiteration}
    Let $f_{j} \colon X \to X \ (j \in \natural)$ be a sequence of contractive mappings on a complete metric space $(X, \rho)$ with an uniform contraction constant $c \in (0, 1)$ and $z_{j} \in X \ (j \in \natural)$ be the unique fixed point of $f_{j}$. 
    Suppose that there exists $x \in X$ such that 
    \begin{equation} \label{Cauchyfinitenesscondition}
        \sum_{j \in \mathbb{N}} c ^{j} \rho(x, z_{j}) < \infty. 
    \end{equation}
    Then, 
    for all $m \in \natural$, there exists $x_{m} \in X$ such that for all $y \in X$ and $n \in \naturalz$,  
    \begin{align} \label{basic_rate_of_naifs}
        \rho(f_{[m, m+n]}(y), x_{m})
        \leq (1+c) \cdot c^{-m} l_{y}(m+n+1), 
    \end{align}
    where $l_{y}(n) := \sum_{k = n}^{\infty} c^{k} \rho(y, z_{k})$ \ $(n \in \natural)$. 
    In addition, $\{x_{m}\}_{m \in \natural}$ has the following properties: 
    \begin{enumerate}
        \item $\{x_{m}\}_{m \in \natural}$ is a recursively compatible sequence for $\{ f_{j} \}_{j \in \natural}$ and
        \item there exists $C > 0$ such that $\rho(y, x_{m}) \leq C \cdot l_{y}(m) \cdot c^{-m}$ for each $y \in X$ and $m \in \natural$. 
    \end{enumerate}
    Moreover, the sequence $\{ x_{m} \}_{m \in \natural}$ with the above properties is unique. 
\end{lemma}
\begin{remark} \label{remark_of_e_contractiveiteration}
    Note that it is necessary to assume that $\sum_{j \in \mathbb{N}} c ^{j} \rho(x, z_{j}) < \infty$ for some $x \in X$ (see, Example \ref{1dimex}). 
    Under the condition, we obtain the inequality (\ref{basic_rate_of_naifs}) which shows that for all $m \in \mathbb{N}$, the sequence $\{f_{[m, m+n]}(y)\}_{n \in  \naturalz}$ converges to $x_{m} \in X$ as $n$ tends to infinity and the limit point $x_{m} \in X$ does not depend on the starting point $y \in X$ (but the convergence rate $l_{y}(n)$ depends on the starting point $y \in X$).
\end{remark}
\begin{proof}[proof of Lemma \ref{e_contractiveiteration}]
    We first show that for each $m \in \natural$, there exists $x_{m} \in X$ such that the inequality (\ref{basic_rate_of_naifs}) holds for all $y \in X$ and $n \in \naturalz$. 
    To show this, let $m \in \natural$ and we set $x_{m}(n) := f_{[m, m+n]}(x) \ (n \in \naturalz)$. 
    Then, by Lemma \ref{collage}, for all $n_{1}, n_{2} \in \naturalz$ with $n_{1} < n_{2}$,
    \begin{align}
        \rho(x_{m}(n_{1}), x_{m}(n_{2})) 
        &\leq \sum_{k=n_{1}+1}^{n_{2}} \rho(x_{m}(k-1), x_{m}(k))  
        = \sum_{k=n_{1}+1}^{n_{2}} \rho(f_{[m, m+k-1]}(x), f_{[m, m+k]}(x)) \notag \\
        &= \sum_{k=n_{1}+1}^{n_{2}} c^{k} \rho(x, f_{m+k}(x)) 
        \leq c^{-m} (1+c) \sum_{k=n_{1}+1}^{n_{2}} c^{k+m} \rho(x, z_{m+k}) \notag \\
        &\leq c^{-m} (1+c) \sum_{k=m+n_{1}+1}^{m+n_{2}} c^{k} \rho(x, z_{k}). \label{importantpoint}
    \end{align}
    It follows that $\{ x_{m}(n) \}_{n \in \naturalz}$ is a Cauchy sequence in $X$ and there exists $x_{m} \in X$ such that $x_{m}(n)$ converges to $ x_{m}$ as $n$ tends to infinity. 

    Now, let $y \in X$. 
    We set $y_{m}(n) := f_{[m, m+n]}(y) \ (n \in \naturalz)$. 
    Note that by Lemma \ref{Cauchyfiniteness}, 
    \begin{equation*}
        \sum_{j \in \mathbb{N}} c ^{j} \rho(y, z_{j}) < \infty.  
    \end{equation*}
    By the same argument, we have 
    \begin{align} \label{importantpointgeneral}
        \rho(y_{m}(n_{1}), y_{m}(n_{2})) \leq c^{-m} (1+c) \sum_{k=m+n_{1}+1}^{m+n_{2}} c^{k} \rho(y, z_{k}) 
    \end{align}
    for all $n_{1}, n_{2} \in \naturalz$ with $n_{1} < n_{2}$ , and there exists $y_{m} \in X$ such that $\{ y_{m}(n) \}_{n \in \naturalz}$ converges to $y_{m}$ as $n$ tends to infinity. 
    In addition, for all $n \in \naturalz$, we have
    \begin{align*}
        \rho(x_{m}, y_{m}) 
        &\leq \rho(x_{m}, x_{m}(n)) + \rho(x_{m}(n), y_{m}(n)) + \rho(y_{m}(n), y_{m}) \\
        &\leq \rho(x_{m}, x_{m}(n)) + c^{n+1}\rho(x, y) + \rho(y_{m}(n), y_{m}). 
    \end{align*}
    It follows that $x_{m} = y_{m}$ for each $m \in \natural$. 
    Besides, by the inequality (\ref{importantpointgeneral}), we have
    \begin{align*} 
        &\rho(f_{[m, m+n_{1}]}(y), x_{m})
        =\lim_{n_{2} \to \infty} \rho(y_{m}(n_{1}), y_{m}(n_{2}))
        \leq c^{-m}(1+c)\sum_{k = m+n_{1}+1}^{\infty} c^{k}\rho(y, z_{k}). 
    \end{align*}
    Therefore, we have proved the inequality (\ref{basic_rate_of_naifs}). 
    
    We next show the properties (i) and (ii) in Lemma \ref{e_contractiveiteration}. 
    Indeed, for all $m \in \natural$, $n \in \naturalz$, 
    \begin{align*}
        f_{m}(x_{m+1}(n))
        = f_{m} \circ f_{[m+1, m+1+n]}(x) 
        = f_{[m, m+n+1]}(x) 
        = x_{m}(n+1). 
    \end{align*}
    Since $x_{m+1}(n) \overset{n \longrightarrow \infty}{\longrightarrow} x_{m+1}$, $x_{m}(n+1)  \overset{n \longrightarrow \infty}{\longrightarrow} x_{m}$ and $f_{m}$ is continuous, we deduce that $f_{m}(x_{m+1}) = x_{m}$ for each $m \in \natural$. 
    In addition, we set $C:= 1+c \ ( > 0 )$ and let $m \in \natural$ and $y \in X$. 
    By the inequality (\ref{importantpointgeneral}) with $n_{1} = 0$, we have
    \begin{align} \label{asym_bdd_of_x_m}
        &c^{m} \rho(y, x_{m})
        = \lim_{n_{2} \to \infty} c^{m} \rho(y, y_{m}(n_{2}))
        \leq \lim_{n_{2} \to \infty} c^{m} \left\{ \rho(y, f_{m}(y)) + \rho(y_{m}(0), y_{m}(n_{2})) \right\} \notag \\
        &\leq c^{m} \left\{ (1+c)\rho(y, z_{m}) + c^{-m}(1+c) \sum_{k = m+1}^{\infty} c^{k}\rho(y, z_{k}) \right\}
        = (1+c)\sum_{k = m}^{\infty} c^{k}\rho(y, z_{k}). 
    \end{align}
    Thus, we have proved the properties (i) and (ii) in Lemma \ref{e_contractiveiteration}. 
    
    We finally show the uniqueness of the sequence $\{ x_{m} \}_{m \in \natural}$ with the properties (i) and (ii) in Lemma \ref{e_contractiveiteration}. 
    Let $m \in \natural$ and $\{ \tilde{x}_{m} \}_{m \in \natural}$ be a sequence in $X$ with the properties. 
    Then, by the properties (i) and (ii) for $\{x_{m}\}_{m \in \natural}$ and $\{\tilde{x}_{m}\}_{m \in \natural}$, there exists $\tilde{C} > 0$ such that, for each $n_{1} \in \naturalz$ and $y \in \naturalz$, 
    \begin{align*}
        \rho(x_{m}, \tilde{x}_{m})
        &\leq \rho(f_{[m, m+n_{1}]}(x_{m}), f_{[m, m+n_{1}]}(y)) + \rho(f_{[m, m+n_{1}]}(y), f_{[m, m+n_{1}]}(\tilde{x}_{m+n_{1}+1})) \\
        &\leq c^{-m} \cdot c^{m+n_{1}+1} \rho (y, x_{m+n_{1}+1}) + c^{-m} \cdot c^{m+n_{1}+1} \rho (y, \tilde{x}_{m+n_{1}+1}) \\
        &= c^{-m} \max \{C, \tilde{C}\} \cdot l_{y}(m+n_{1}+1). 
    \end{align*}
    It follows that $x_{m} = x^{\prime}_{m}$ for each $m \in \natural$. 
    Hence, we have proved our lemma. 
\end{proof}
As we have mentioned in Remark \ref{remark_of_e_contractiveiteration}, the rate $l_{y}$ in Lemma \ref{e_contractiveiteration} depends on the starting point $y \in X$ of non-autonomous iterations. 
To obtain more similar results to Banach's fixed point theorem, we need the following Lemma. 
\begin{lemma}\label{further_e_contractiveiteration}
    Let $f_{j} \colon X \to X \ (j \in \natural)$ be a sequence of contractive mappings on a complete metric space $(X, \rho)$ with an uniform contraction constant $c \in (0, 1)$ and $z_{j} \in X \ (j \in \natural)$ be the unique fixed point of $f_{j}$. 
    Suppose that there exist $x \in X$ and $a \colon \natural \to \real$ with $\sum_{j \in \natural} a(j) < \infty$ such that for all $j \in \mathbb{N}$, 
    \begin{equation} \label{afastconvcondition}
        c^{j} \rho(x, z_{j}) \leq a(j).    
    \end{equation}
    Then, for all $m \in \natural$, there exists $x_{m} \in X$ such that for all $y \in X$ and $n \in \naturalz$,  
    \begin{align} \label{yajrate}
        \rho(f_{[m, m+n]}(y), x_{m})
        \leq \max\{ (1+c), \rho(y, x)\} 
        \cdot c^{-m} \cdot l^{\prime}(m+n+1), 
    \end{align}
    where $l^{\prime}(n) := \max \{ c^{n}, \sum_{k=n}^{\infty} a(k) \} \ (n \in \natural)$. 
    In addition, $\{x_{m}\}_{m \in \natural}$ have the following properties: 
    \begin{enumerate}
        \item $\{x_{m}\}_{m \in \natural}$ is a recursively compatible sequence for $\{ f_{j} \}_{j \in \natural}$ and
        \item for each $y \in X$, there exists $C(y) > 0$ such that for each $m \in \natural$, $$\rho(y, x_{m}) \leq C(y) \cdot \max \left\{ 1, \  c^{-m}\sum_{k=m}^{\infty} a(k) \right\} \left( = C(y) \cdot l^{\prime}(m) \cdot c^{-m} \right). $$
    \end{enumerate}
    Moreover, the sequence $\{ x_{m} \}_{m \in \natural}$ with the above properties is unique. 
\end{lemma}
\begin{remark}\label{remarl_of_further_e_contractiveiteration}
    The inequality (\ref{yajrate}) shows that for all $m \in \mathbb{N}$, $\{f_{[m, m+n]}(y)\}_{n \in  \naturalz}$ converges to $x_{m} \in X$ as $n$ tends to infinity with the rate $l^{\prime}$. 
    In addition, the limit point $x_{m} \in X$ does not depend on the starting point $y \in X$ and the starting point depends on only the constant of the convergence rate. 
    However, the convergence is not always a exponentially fast rate (see, Example \ref{polyfastconvexample} and Example \ref{unboundedex}). 
\end{remark}
\begin{proof}[proof of Lemma \ref{further_e_contractiveiteration}]
    We first show that the inequality (\ref{yajrate}). 
    By the assumption (\ref{afastconvcondition}), we have $\sum_{j =1}^{\infty} c^{j} \rho(x, z_{j}) \leq \sum_{j = 1}^{\infty} a(j) < \infty$
    and the condition in Lemma \ref{e_contractiveiteration} is satisfied. 
    By the inequality (\ref{basic_rate_of_naifs}) with $y = x$, we obtain that
    \begin{align*}
        \rho(f_{[m, m+n]}(y), x_{m})
        &\leq \rho(f_{[m, m+n]}(y), f_{[m, m+n]}(x)) + \rho(f_{[m, m+n]}(x), x_{m}) \\
        &\leq c^{-m} \left\{ c^{m+n+1} \cdot \rho(y, x) + (1+c) \sum_{k = m+n+1}^{\infty} a(k) \right\} \\
        &\leq \max\{ (1+c), \rho(y, x)\} 
        \cdot c^{-m} \cdot \max \left\{ c^{m+n+1}, \sum_{k=m+n+1}^{\infty} a(k) \right\}. 
    \end{align*}
    Therefore, we have proved the inequality (\ref{yajrate}). 

    We next show the properties (i) and (ii) in Lemma \ref{further_e_contractiveiteration}. 
    Since we have shown that the condition in Lemma \ref{e_contractiveiteration} is satisfied, we have already proved the properties (i) and (ii) in Lemma \ref{e_contractiveiteration} (or the inequality (\ref{asym_bdd_of_x_m})) for each $m \in \natural$ and $y \in X$. 
    By the inequality (\ref{asym_bdd_of_x_m}) with $y = x$, it follows that
    \begin{align}
        c^{m} \rho(y, x_{m})
        &\leq c^{m}\rho(y, x) + c^{m}\rho(x, x_{m})
        \leq \rho(y, x) \cdot c^{m} + (1+c) \sum_{k=m}^{\infty} a(k) \notag \\
        &\leq \max\{\rho(y, x), (1 + c) \} \cdot \max\left\{ c^{m},  \sum_{k=m}^{\infty} a(k) \right\} \label{a_fast_bdd}. 
    \end{align}
    Thus, we have proved the properties (i) and (ii) in Lemma \ref{further_e_contractiveiteration}. 
    
    Finally, by the same argument in the proof of the uniqueness of the sequence $\{ x_{m} \}_{m \in \natural}$ in Lemma \ref{e_contractiveiteration}, we also deduce that the uniqueness of the sequence $\{ x_{m} \}_{m \in \natural}$. 
    Hence, we have proved our lemma. 
\end{proof}
\begin{remark} \label{expfast_remark}
    In Lemma \ref{further_e_contractiveiteration}, if $a(j) = C^{\prime} \cdot r^{j}$ for some $r \in \left[ c, 1 \right)$ and $C^{\prime} > 0$, then the condition in Lemma \ref{e_contractiveiteration} is satisfied and we obtain the following: 
    for all $m \in \mathbb{N}$, there exists $x_{m} \in X$ such that for all $y \in X$, the sequence $\{f_{[m, m+n]}(y)\}_{n \in  \naturalz}$ converges to $x_{m}$ as $n$ tends to infinity exponentially fast with the rate $r$.  
    In addition, $\{x_{m}\}_{m \in \natural}$ has the following properties: 
    \begin{enumerate}
        \item $\{x_{m}\}_{m \in \natural}$ is a recursively compatible sequence for $\{ f_{j} \}_{j \in \natural}$ and
        \item $(c/r)^{m}\rho(y, x_{m}) \ (m \in \natural)$ is bounded for each $y \in X$
        ( equivalently, for each $y \in X$, $c^{m}\rho(y, x_{m}) \to 0$ as $m \to \infty$ exponentially fast with the rate $r$). 
    \end{enumerate}
    Moreover, the sequence $\{ x_{m} \}_{m \in \natural}$ with the above properties is unique. 
    
    Indeed, by inequalities (\ref{basic_rate_of_naifs}) with $y = x$ and the assumption $a(j) = C^{\prime} \cdot r^{j} \ (j \in \natural)$, there exists $x_{m} \in X \ (m \in \natural)$ such that for all $y \in X$ and $n \in \naturalz$, 
    \begin{align}
        \rho(f_{[m, m+n]}(y), x_{m})
        &\leq \rho(f_{[m, m+n]}(y), f_{[m, m+n]}(x)) + \rho(f_{[m, m+n]}(x), x_{m}) \notag \\
        &\leq  c^{n+1} \cdot \rho(y, x) + c^{-m} \cdot (1+c) C^{\prime} \sum_{k = m+n_{1}+1}^{\infty}  r^{k} \notag \\
        &\leq \max \left\{ \rho(y, x), \frac{1+c}{1-r} C^{\prime} \left( \frac{r}{c} \right)^{m} \right\} \cdot r^{n+1}.  \label{exp_basic_estimate}  
    \end{align}
    
    In addition, since the condition in Lemma \ref{e_contractiveiteration} is satisfied, we have already proved the properties (i).  
    Moreover, by the property (ii) in Lemma \ref{e_contractiveiteration} (or the inequality (\ref{asym_bdd_of_x_m}) )
    and the assumption $a(j) = C^{\prime} \cdot r^{j} \ (j \in \natural)$, it follows that
    \begin{align} 
        c^{m} \rho(y, x_{m})
        &\leq c^{m} \rho(y, x) + c^{m} \rho(x, x_{m})
        \leq c^{m} \rho(y, x) + (1+c) C^{\prime} \sum_{k = m}^{\infty}  r^{k} \notag \\
        &\leq \max \left\{ \rho(y, x),  C^{\prime} \frac{1+c}{1-r} \right\} \cdot r^{m}, \label{exp_fast_bdd}
    \end{align}
    which shows that $(c/r)^{m}\rho(y, x_{m}) \ (m \in \natural)$ is bounded for each $y \in X$. 
    Finally, uniqueness of $\{x_{m}\}_{m \in \natural}$ is deduced by the same argument in Lemma \ref{further_e_contractiveiteration} and inequality (\ref{exp_fast_bdd}). 
\end{remark}
\begin{remark} \label{bounded_case_remark}
    If $Z$ is unbounded, 
    then $\{ x_{m} \}_{m \in \natural}$ is also unbounded in general even if we assume the condition in Lemmas \ref{e_contractiveiteration}, \ref{further_e_contractiveiteration} or \ref{expfast_remark}
    (see Example \ref{unboundedex}). 
    On the other hand, if $Z$ is bounded, then the condition in Remark \ref{expfast_remark} is automatically satisfied with $r := c$ and $C^{\prime} := \sup_{j \in \natural} \rho(x, z_{j})$, and we deduce that the unique recursively compatible sequence $\{ x_{m} \}_{m \in \natural}$ is bounded by the property (ii) in Remark \ref{expfast_remark} (or the inequality (\ref{exp_fast_bdd})).
    
    In particular, if $X$ is bounded or $\{ f_{j} \}_{j \in \natural}$ is autonomous (i.e. $f_{j} := f_{1}$ for all $j \in \natural$), then the condition in Remark \ref{expfast_remark} is automatically satisfied. 
    Note that if $\{ f_{j} \}_{j \in \natural}$ is autonomous, then the unique recursively compatible sequence is the constant sequence of fixed point of $f_{1}$ (i.e. $\{x_{m}\}_{m \in \natural} = \{ z_{1} \}_{m \in \natural}$). 
\end{remark}
Before we conclude this subsection, we show the corollary of Remark \ref{expfast_remark}.   
\begin{corollary} \label{expfastsufficond}
    Let $\{ f_{j} \}_{j \in \natural}$ be a sequence of contractive mappings on a complete metric space $(X, \rho)$ with an uniformly contraction constant $c \in (0, 1)$. 
    For each $j \in \natural$, let $z_{j} \in X$ be the unique fixed point of $f_{j}$. 
    Suppose that there exists $x \in X$, 
    \begin{equation} 
        a := \limsup_{j \longrightarrow \infty} \sqrt[j]{\rho(x, z_{j})} < \frac{1}{c}. \label{limsupcondition}
    \end{equation}
    Then, for all $m \in \natural$ and $r \in \{ r > 0 \ | \ c \leq r < 1, ac < r \}$, there exists $x_{m} \in X$ such that for all $y \in X$, the sequence $\{f_{[m, m+n]}(y)\}_{n \in  \naturalz}$ converges to $x_{m}$ as $n$ tends to infinity exponentially fast with the rate $r$.  
    In addition, $\{x_{m}\}_{m \in \natural}$ has the following properties: 
    \begin{enumerate}
        \item $\{x_{m}\}_{m \in \natural}$ is a recursively compatible sequence for $\{ f_{j} \}_{j \in \natural}$ and
        \item $(c/r)^{m}\rho(y, x_{m}) \ (m \in \natural)$ is bounded for each $y \in X$. 
    \end{enumerate}
    Moreover, the sequence $\{ x_{m} \}_{m \in \natural}$ with the above properties is unique. 
\end{corollary}
\begin{proof}
    Let $r \in \{ r > 0 \ | \ c \leq r < 1, ac < r \}$. 
    Then, by the assumption (\ref{limsupcondition}), there exists $M \in \mathbb{N}$ such that for all $j \geq M$, we have $c \cdot \sqrt[j]{\rho(x, z_{j})} < r$, which is equivalent to $c^{j} \rho(x, z_{j}) < r^{j}$. 
    Therefore, we have 
    \begin{align*}
        c^{j} \rho(x^{\prime}, z_{j}) \leq C^{\prime} r^{j}
    \end{align*}
    for all $j \in \mathbb{N}$, 
    where $C^{\prime} := \max ( \{ \ c^{j} \rho(x, z_{j}) / r_{j} | \ j < M \ \} \cup \{ 1 \} ) ( > 0)$. 
    It follow that the condition in Remark \ref{expfast_remark} is satisfied and our statement of the corollary holds.
\end{proof}
\begin{remark}\label{independence_of_alpha}
    The constant $\alpha \geq 0$ in Corollary \ref{expfastsufficond} does not depend on $x \in X$ if $Z$ is unbounded. 
    To this end, let $x, y \in X$ with $y \neq x$. 
    Note that, since $Z$ is unbounded, sequences $\{ \rho(x, z_{j}) \}_{j \in \natural}$ and $\{ \rho(y, z_{j}) \}_{j \in \natural}$ are unbounded and we deduce that
    \begin{align*}
        \alpha := \limsup_{j \longrightarrow \infty} \sqrt[j]{\rho(x, z_{j})} \geq 1 \quad \text{and} \quad 
        \alpha^{\prime} := \limsup_{j \longrightarrow \infty} \sqrt[j]{\rho(y, z_{j})} \geq 1. 
    \end{align*}
    We first show that $\alpha^{\prime} > 1$ if $\alpha > 1$. 
    Note that there exists a subsequence $n_{k} \in \natural \ (k \in \natural)$ such that $\sqrt[n_{k}]{\rho(x, z_{n_{k}})}$ converges to $\alpha$ as $k$ tends to infinity with the following properties:  
    \begin{align*}
        \sqrt[n_{k}]{\rho(x, z_{n_{k}})} \geq \alpha \quad \text{and} \quad
        \alpha^{n_{k}} \geq \rho(x, y) + 1 \quad \text{for each} \ k \in \natural. 
    \end{align*}
    Since $\rho(y, z_{n_{k}}) \geq \rho(x, z_{n_{k}}) - \rho(x, y) \geq \alpha^{n_{k}} - \rho(x, y) \geq 1$ for each $k \in \natural$, 
    we deduce that
    \begin{align*}
        \frac{1}{n_{k}} \log \rho(x, z_{n_{k}})
        &\leq \frac{1}{n_{k}} \log \left\{ \rho(y, z_{n_{k}}) + \rho(x, y) \right\}
        \leq \frac{1}{n_{k}} \log \rho(y, z_{n_{k}}) + \frac{1}{n_{k}} \log \left( 1 + \rho(x, y) \right)
    \end{align*}
    Now, let $N \in \natural$. 
    Since there exists $K \in \natural$ such that $n_{k} \geq N$ for each $k \in \natural$ with $k \geq K$, we obtain that
    \begin{align*}
        \frac{1}{n_{k}} \log \rho(x, z_{n_{k}})
        &\leq \sup_{n \geq N} \frac{1}{n}\log \rho(y, z_{n}) + \frac{1}{N} \log \left( 1 + \rho(x, y) \right)
    \end{align*}
    and it follows that $\log \alpha \leq \log \alpha^{\prime}$. 
    Therefore, we have proved that $\alpha^{\prime} > 1$ if $\alpha > 1$. 
    
    We next show that $\alpha^{\prime} = 1$ if $\alpha = 1$. 
    Indeed, since it is sufficient to show that $\alpha^{\prime} \leq 1$, we assume that $\alpha^{\prime} > 1$. 
    Then, by switching $x$ and $y$ in the above argument, we obtain that $\alpha > 1$ and this contradicts $\alpha = 1$.   
    Thus, we have proved that $\alpha^{\prime} = 1$ if $\alpha = 1$. 
    
    On the other hand, the constant $\alpha \geq 0$ depends $x \in X$ if $Z$ is bounded. 
    In fact, if we consider an autonomous iteration on a complete metric space $X$ (and assume that $X$ is not a single set), then we easily show that $\alpha = 0$ or $1$. 
    However, as we have mentioned in Remark \ref{bounded_case_remark}, the condition in Remark \ref{expfast_remark} is automatically satisfied with $r := c$ and $C^{\prime} := \sup_{j \in \natural} \rho(x, z_{j})$ if $Z$ is bounded. 
    Therefore, by Remark \ref{expfast_remark}, we have already obtained a ``compatible" result than Corollary \ref{expfastsufficond} if $Z$ is bounded.  
\end{remark}
\subsection{Examples of the sequence of contractive mappings} 
In this subsection, we consider some examples of non-autonomous iterations. 
The following examples show that non-autonomous iterations have different properties from autonomous iterations and non-autonomous iterations on a bounded set.  

The following example shows that if we do not assume the condition in Lemma \ref{e_contractiveiteration}, the conclusion in Lemma \ref{e_contractiveiteration} does not hold in general. 
\begin{example}\label{1dimex}
    Let $f_{j} \colon \real \to \real \ (j \in \natural)$ are defined by 
    \begin{equation*}
        f_{j}(x) := c(x -a_{j} ) + a_{j} = cx + (1-c) a_{j} \ (x \in \real ),  
    \end{equation*}
    where $c \in (0, 1)$ and $a_{j} \in \real$. 
    Note that for each $j \in \mathbb{N}$, $a_{j} \in \real$ is the unique fixed point of $f_{j}$ i.e. $z_{j} = a_{j}$. 
    Then, 
    by induction with respect to $n \in \natural$, 
    for all $m \in \natural$, $n \in \naturalz$ and $x \in \mathbb{R}$, we have
    \begin{equation} \label{similitudeformula}
        f_{[m, m+n]}(x) = c^{n+1}x + (1-c) \sum_{j=0}^{n} c^{j}a_{m+j}. 
    \end{equation}
%

    Now, we consider the example which does not satisfies the assumption in Lemma \ref{e_contractiveiteration}. 
    Let $c = 1/2$, $a_{i} = 2^{i+1}$. 
    Note that 
    \begin{align*}
        \sum_{j = 1}^{\infty} c^{j} \cdot \rho_{\real}(0, a_{j}) 
        = \sum_{j = 1}^{\infty} 2^{-j} \cdot 2^{j+1} 
        = \sum_{j = 1}^{\infty} 2 = \infty,  
    \end{align*}
    where $\rho_{\real}$ is the Euclidean metric on $\real$. 
    It follows that this example does not satisfies the assumption in Lemma \ref{e_contractiveiteration}.  
    In addition, by the equation (\ref{similitudeformula}), we have 
    \begin{equation*}
        f_{[m, m+n]}(x) = \frac{1}{2^{n+1}} x + \frac{1}{2} \sum_{j=0}^{n} \frac{1}{2^{j}} \cdot 2^{m+j} = \frac{1}{2^{k}} x + 2^{m} \cdot (n+1)
    \end{equation*}
    for each $x \in \real$, $m \in \natural$ and $n \in \naturalz$.  
    Therefore, we obtain that for each $m \in \mathbb{N}$ and $x \in \mathbb{R}$, $f_{[m, m+n]}(x)$ does not converge as $n$ tends to infinity, which deduce that the conclusion in Lemma \ref{e_contractiveiteration} does not hold without the condition in Lemma \ref{e_contractiveiteration}. 
\end{example}

The following example shows that there is a example which satisfies the condition in Lemma \ref{e_contractiveiteration} ( the inequality (\ref{Cauchyfiniteness})) but which does not satisfies the condition in Remark \ref{expfast_remark}. 
In addition, we estimate the convergence rate of the non-autonomous iteration and show that recursively compatible sequence is unbounded. 
\begin{example}\label{polyfastconvexample}
    In Example \ref{1dimex}, we set $c = 1/2$, $a_{i} = 2^{i}/i^{l+1} \ (l \in \natural)$. 
    Then, note that 
    \begin{equation}\label{polysumcalc}
        \sum_{j = 1}^{\infty} c^{j} \cdot \rho_{\real}(0, z_{j}) 
        = \sum_{j = 1}^{\infty} \frac{1}{2^{j}} \cdot \frac{2^{j}}{j^{l+1}} 
        = \sum_{j = 1}^{\infty} \frac{1}{j^{l+1}}
        < \infty,   
    \end{equation}
    where $\rho_{\real}$ is the Euclidean metric on $\real$. 

    Then, we first show that this example does not satisfies the condition in Remark \ref{expfast_remark}. 
    Indeed, let $r \in (0, 1)$. 
    If for each $m \in \natural$ and $y \in X$, there exists and $\tilde{C}_{m}(y) > 0$ such that 
    \begin{align*}
        \rho_{\real}(f_{[m, m+n]}(x), x_{m}))
        \leq \tilde{C}_{m}(y) \cdot r^{n}
    \end{align*}
    for each $n \in \naturalz$. 
    Then, by the equation (\ref{similitudeformula}), we have 
    \begin{align*}
        f_{[m, m+n]}(x) 
        = \frac{1}{2} \sum_{k = 0}^{n} \frac{1}{2^{k}} \frac{2^{m+k}}{(m+k)^{l+1}}
        = 2^{m-1} \cdot \sum_{k = 0}^{n} \frac{1}{(m+k)^{l+1}}
        = 2^{m-1} \cdot \sum_{k = m}^{m+n} \frac{1}{k^{l+1}}. 
    \end{align*}
    Since $x_{m} = \lim_{n \to \infty} f_{[m, m+n]}(x) = 2^{m-1} \cdot \sum_{k = m}^{\infty} 1/k^{l+1}$, we deduce that
    \begin{align*}
        \frac{1}{(m+n+1)^{l+1}} 
        < 2^{m-1} \cdot \sum_{k = m+n+1}^{\infty} \frac{1}{k^{l+1}} 
        = \rho_{\real}(f_{[m, m+n]}(x), x_{m})
        \leq \tilde{C}_{m}(y) \cdot r^{n}
    \end{align*}
    for each $n \in \naturalz$. 
    it follows that $1 < \tilde{C}_{m}(y) (m+n+1)^{l+1} r^{n} \to 0$ as $n$ tends to infinity, which deduce the contradiction. 
    Therefore, we have proved that this example does not satisfies the condition in Remark \ref{expfast_remark}. 
    Instead, let $m \in \natural$. Since 
    \begin{align*}
        c^{j} \rho_{\real}(x, z_{j}) = 1/2^{j} \cdot 2^{j}/j^{l+1} = 1/j^{l+1}
    \end{align*}
    for each $j \in \natural$ and $\sum_{j \in \natural} 1/j^{l+1} < \infty$, we obtain that this example satisfies the condition in Lemma \ref{further_e_contractiveiteration} with $a(j) := 1/j^{l+1} \ (j \in \natural)$. 
    Moreover, we have $k^{l-1} + k^{l} \leq (k+1)^{l}$ and 
    \begin{align*}
        \frac{1}{(k+1)^{l+1}} 
        < \frac{1}{k(k+1)^{l}} 
        = \frac{k^{l-1}}{k^{l}(k+1)^{l}} 
        \leq \frac{(k+1)^{l} - k^{l}}{k^{l}(k+1)^{l+1}}
        = \frac{1}{k^{l}} - \frac{1}{(k+1)^{l}}
    \end{align*}
    for each $k, l \in \natural$. 
    It follows that 
    \begin{align*}
        c^{-m} \cdot \sum_{k = m+n+1}^{\infty} a(k)
        = 2^{m} \cdot \sum_{k = m+n}^{\infty} \frac{1}{(k+1)^{l+1}}
        \leq 2^{m} \cdot \sum_{k = m+n}^{\infty} \left\{ \frac{1}{k^{l}} - \frac{1}{(k+1)^{l}} \right\}
        =  \frac{2^{m}}{(m+n)^{l}}
    \end{align*}
    for each $m \in \natural$ and $n \in \naturalz$. 
    By Lemma \ref{further_e_contractiveiteration}, we deduce that
    \begin{align*}
        \rho_{\real}(f_{[m, m+n]}(y), x_{m})
        \leq \max\{ (1+c), \rho(y, x)\} 
        \cdot \max \left\{ \frac{1}{2^{n+1}}, \frac{2^{m}}{(m+n)^{l}} \right\} 
    \end{align*}
    for each $m \in \natural$, $n \in \naturalz$ and $y \in \real$.  
    Thus, we obtain the convergence rate of the sequence $\{ f_{[m, m+n]}(y) \}_{n \in \natural}$ in this example. 
    It is obvious that $\{ x_{m} \in X \ | \ m \in \mathbb{N} \}$ is unbounded since $x_{m} = 2^{m-1} \cdot \sum_{k = m}^{\infty} 1/k^{l+1}$ for each $m \in \natural$. 
\end{example}

The following example shows that the set $\{ x_{m} \in X \ | \ m \in \mathbb{N} \}$ is  unbounded in general even if we assume that the condition in Remark \ref{expfast_remark} holds. 
\begin{example}\label{unboundedex}
    In Example \ref{1dimex}, we set $c \in (0, 1)$, $a_{i} =i$ and $x=0 \in \real$. 
    Note that $\{ z_{j} \ | \ j \in \mathbb{N} \} = \{ a_{j} \ | \ j \in \mathbb{N} \} = \mathbb{N}$ is unbounded and there exist $C > 0$ and $r \in (c, 1)$ such that  
    \begin{equation}
        \rho_{\real}(x, z_{j}) = j \leq C \cdot \left( \frac{r}{c} \right)^{j}
    \end{equation}
    for each $j \in \natural$, where $\rho_{\real}$ is the Euclidean metric on $\real$. 
    Therefore, the non-autonomous iterations $\{f_{j}\}_{j \in \natural}$ satisfies the condition in Remark \ref{expfast_remark}. 
    In addition, let $n \in \natural$ and we set $S_{n} := \sum_{j=0}^{n} j \cdot c^{j} = \sum_{j=1}^{n} j \cdot c^{j}$. 
    Then, we have
    \begin{align*}
        (1-c) S_{n}
        &= \sum_{j=1}^{n} j \cdot c^{j} - \sum_{j=1}^{n} j \cdot c^{j+1} 
        = c + \sum_{j=1}^{n-1} (j+1) \cdot c^{j+1} - \sum_{j=1}^{n-1} j \cdot c^{j+1} - n \cdot c^{n+1} \\
        &= c + \sum_{j=1}^{n-1} c^{j+1} - n \cdot c^{n+1} 
        =  c + (1-c)^{-1} \left( c^{2} - c^{n+1} \right) - n \cdot c^{n+1}.  
    \end{align*}
    By the equation (\ref{similitudeformula}), it follows that 
    \begin{align*}
        &f_{[m, m+n]}(y) 
        = c^{n+1}y + (1-c) \sum_{j=0}^{n} c^{j} \cdot (m+j) 
        = c^{n+1}y + (1-c)m \sum_{j=0}^{n} c^{j} + (1-c) S_{n} \\ 
        &= c^{n+1}y + m \left( 1 - c^{n+1} \right) + c + (1-c)^{-1} \left( c^{2} - c^{n+1} \right) - n \cdot c^{n+1}
        \longrightarrow m + c + c^{2} (1-c)^{-1} 
    \end{align*}
    as $n$ tends to infinity
    for all $m \in \natural$ and $y \in \mathbb{R}$. 
    Thus, we deduce that the set $\{ x_{m} \in X \ | \ m \in \mathbb{N} \}$ is also unbounded.  
\end{example}
%
%
%
%
%
%
%
%
%
%
%
%
%
%
%
%
%
%
%
%
%
%
%
%
%
%
%
%
%
%
%
%
%
%
%
%
%
%
%
%
%
%
\section{General iterated function systems} \label{general_IFS}
In this section, we present the definition and properties of general IFSs which are the main notion of this paper. 
In Subsection 3.1, we first introduce the notion of words and trees which we need when we define the general IFSs. 
In Subsection 3.2, we next introduce the definition of general IFSs and consider the projection map for the general IFSs. 

%
%
%
%
%
\subsection{Words and trees} \label{word_edgemap_tree}
Let $I$ be a countable set and $\mathcal{J} := \{ J \subset I \ | \ \#(J) < \infty \ \}$, 
where $\#(A)$ is the cardinality of $A$ for each set $A$. 
Besides, we set $I^{*} := \{ \phi \} \cup \bigcup_{n \in \natural} I^{n}$, where $\phi$ is not a element of $I$. 
We write $\omega \in I^{m} \ (m \in \natural)$ as $\omega_{1} \cdots \omega_{m} \ (\omega_{k} \in I, k = 1, \ldots , m)$ and $\omega \in I^{\natural}$ as $\omega_{1} \omega_{2} \cdots\ (\omega_{k} \in I, k \in \natural)$ respectively. 
For each $\omega \in I^{*} \cup I^{\natural} $, we set 
\begin{align*}
    |\omega| := 
    \begin{cases}
        0       & \text{if} \ \omega = \phi \\
        n       & \text{if} \ \omega \in I^{n} \ (n \in \natural) \\
        \infty  & \text{if} \ \omega \in I^{\natural}
    \end{cases}. 
\end{align*}
The set $I$ is often called the alphabet and $I^{*}$ is called the set of words with a finite length. 
Moreover, the convolution $\omega\omega^{\prime}$ of $\omega \in I^{*}$ and $\omega^{\prime} \in I^{*} \cup I^{\natural}$ is defined by
\begin{align*}
    \omega\omega^{\prime} := 
    \begin{cases}
        \omega_{1} \cdots \omega_{|\omega|}\omega^{\prime}_{1} \cdots \omega^{\prime}_{|\omega^{\prime}|} & \text{if} \ \omega \in I^{*} \\
        \omega_{1} \cdots \omega_{|\omega|}\omega^{\prime}_{1} \cdots & \text{if} \ \omega \in I^{\natural}
    \end{cases}. 
\end{align*}
    
The maps $\Pi_{[m, n]} \colon I^{\natural} \to I^{n-m+1}$ \ $(m, n \in \natural$ with $n \geq m$) and $\Pi_{[m, n)} \colon I^{\natural} \to I^{n-m}$ \ $(m, n \in \natural$ with $n > m$) are defined by
\begin{align*}
    \Pi_{[m, n]}(\omega) &:= \omega_{m} \cdots \omega_{n} \quad \text{and} \quad
    \Pi_{[m, n)}(\omega) := \omega_{m} \cdots \omega_{n-1} \quad (\omega := \omega_{1}\omega_{2} \cdots \in I^{\natural})
\end{align*}
respectively. 
For simplicity, we write $\Pi_{[n, n]}$ as $\Pi_{n}$, $\Pi_{[m, n]}(\omega)$ as $\omega_{[m, n]}$ and $\Pi_{[m, n)}(\omega)$ as $\omega_{[m, n)}$ respectively. 
Also, the maps $\Pi_{(m, n]} \colon I^{\natural} \to I^{n-m}$ and $\Pi_{(m, n)} \colon I^{\natural} \to I^{n-m-1}$ are defined by 
\begin{align*}
    \Pi_{(m, n]}(\omega) &:= \omega_{m+1} \cdots \omega_{n} \quad \text{and} \quad
    \Pi_{(m, n)}(\omega) := \omega_{m+1} \cdots \omega_{n-1} \quad (\omega := \omega_{1}\omega_{2} \cdots \in I^{\natural})     
\end{align*}
respectively if the relation of $m, n \in \natural$ is compatible.  
Also, for simplicity, we write $\Pi_{(m, n]}(\omega)$ as $\omega_{(m, n]}$ and $\Pi_{(m, n)}(\omega)$ as $\omega_{(m, n)}$ respectively. 
Similarly, the maps $\Pi_{[m, \infty)} \colon I^{\natural} \to I^{\natural}$ and $\Pi_{(m, \infty)} \colon I^{\natural} \to I^{\natural}$ \ ($m \in \natural)$ are defined by
\begin{align*}
    \Pi_{[m, \infty)}(\omega) &:= \omega_{m} \omega_{m+1} \cdots \quad \text{and} \quad \Pi_{(m, \infty)}(\omega) := \omega_{m+1} \omega_{m+2} \cdots  \quad (\omega := \omega_{1}\omega_{2} \cdots \in I^{\natural})  
\end{align*}
respectively. 
Similarly, for simplicity, we write $\Pi_{[m, \infty)}(\omega)$ as $\omega_{[m, \infty)}$ and $\Pi_{(m, \infty)}(\omega)$ as $\omega_{(m, \infty)}$ respectively. 
we endow alphabet $I$ with the discrete topology and $I^{\natural}$ with the product topology. 

We now introduce the definition of trees. 
\begin{definition} \label{def_of_tree}
    Let $I$ be a countable set and $\phi$ be a non-element of $I$. 
    We say that non-empty closed subset $\tree$ of $I^{\natural}$ is a tree with $I$ if $S(\tree, \phi) := \Pi_{1}(\tree) \in \mathcal{J}$, and
    \begin{align*}
        S(\tree, \omega_{[1, n]}) := \{ \tau_{n+1} \in I \ | \ \tau \in \tree, \ \tau_{[1, n]} = \omega_{[1, n]} \ \} \in \mathcal{J}
    \end{align*}
    for each $n \in \natural$ and $\omega \in \tree$. 
    For a tree $\tree$, we set the following: 
    \begin{align*}
        \tree^{0} := \{\phi\}, \quad
        \tree^{[m, n]} := \Pi_{[m, n]}(\tree) \ (m, n \in \natural \ \text{with} \ m \leq n) \quad \text{and} \quad \tree^{[1, *]} := \tree^{0} \cup \bigcup_{n \in \natural} \tree^{[1, n]}.  
    \end{align*}
    Also, we set $\tree^{n} := \tree^{[n, n]} \ (n \in \natural)$ for simplicity. 
    In addition, we set $\mathcal{J}_{1} := \{ \Pi_{1}(\tree) \} \ (\subset \mathcal{J})$ and $I_{1} := \Pi_{1}(\tree) \ (\subset I)$. 
    Besides, for each $n \in \natural$ with $n \geq 2$, we set 
    \begin{align*}
        \mathcal{J}_{n} &:= \{ S(\tree, \omega_{[1, n-1]}) \ | \ \omega_{[1, n-1]} \in \tree^{[1, n-1]} \ \} \ (\subset \mathcal{J}) \ \text{and} \\ 
        I_{n} &:= \bigcup_{\omega_{[1, n-1]} \in \tree^{[1, n-1]}} S(\tree, \omega_{[1, n-1]}) \ ( \subset I)
    \end{align*}
    respectively.  
\end{definition}
    Note that $\#(\tree^{[1, n]}) < \infty$ for each $n \in \natural$. 
%
    In addition, note that $\#(\mathcal{J}_{n}) < \infty$ and $\#(I_{n}) < \infty$ for each $n \in \natural$. 
\begin{remark} \label{definition_subtree}
    By the similar argument, we also define the tree $\tree_{\omega}$ (called the sub-tree of $\tree$ conditioned by $\omega \in \tree^{[1, *]}$ ). 
    Indeed, let $\omega \in \tree^{[1, *]}$ and we set the following:  
    \begin{align*}
        \tree_{\omega} &:= \{ \Pi_{(|\omega|, \infty)}(\tau) \in \Pi_{(|\omega|, \infty)}(\tree) \ | \ \tau \in \tree, \ \Pi_{[1, |\omega|]}(\tau) = \omega \ \}, \quad \tree_{\omega}^{0} := \{\phi\}, \\ 
        \tree_{\omega}^{[m, n]} &:= \Pi_{[m , n]}(\tree_{\omega}) \ (m, n \in \natural \ \text{with} \ m \leq n ) \quad \text{and} \quad 
        \tree_{\omega}^{[1, *]} := \tree_{\omega}^{0} \cup \bigcup_{n \in \natural} \tree_{\omega}^{[1, n]}.  
    \end{align*}
    Also, we set $\tree_{\omega}^{n} := \tree_{\omega}^{[n, n]}$ \ ($n \in \natural$) for simplicity.  
    We endow $\tree_{\omega} \ (\subset I^{\natural})$ with the induced topology. 
    Note that $\tree_{\omega} \subset I^{\natural}$ is a tree for each $\omega \in \tree^{[1, *]}$. 
    To show this, let $\omega \in \tree^{[1, *]}$. 
    It is obvious that $\tree_{\omega}$ is colsed, since $\tree$ is closed and the map $\tau \mapsto \omega \tau$ is continuous on $I^{\natural}$. 
    In addition, Since $\Pi_{[m, n]} \circ \Pi_{(|\omega|, \infty)} = \Pi_{[|\omega|+m, |\omega|+n]}$ for each $m, n \in \natural$ with $m \leq n$, we have 
    \begin{align*}
        S(\tree_{\omega}, \phi)
        =\Pi_{1}(\tree_{\omega}) 
        = \{ \Pi_{|\omega|+1}(\tau) \in I \ | \ \tau \in \tree, \ \Pi_{[1, |\omega|]}(\tau) = \omega \ \} = S(\tree, \omega) \in \mathcal{J}.  
    \end{align*}
    Moreover, let $n \in \natural$ and $\tilde{\omega} \in \tree_{\omega}^{[1, n]}$. 
    Since $|\omega|+ n = |\omega|+|\tilde{\omega}| = |\omega\tilde{\omega}|$, we have
    \begin{align*}
        S(\tree_{\omega}, \omega^{\prime})
        &=\{\Pi_{n+1}(\tau^{\prime}) \in I \ | \ \tau^{\prime} \in \tree_{\omega}, \ \Pi_{[1, n]}(\tau^{\prime}) = \omega^{\prime} \ \} \\
        &= \{\Pi_{n+1}(\Pi_{(|\omega|, \infty)}(\tau)) \in I \ | \ \tau \in \tree, \ \Pi_{[1, |\omega|]}(\tau) = \omega, \ \Pi_{[1, n]}(\Pi_{(|\omega|, \infty)}(\tau)) = \omega^{\prime} \ \} \\
        &= \{\Pi_{[|\omega|+n+1, |\omega|+n+1]}(\tau) \in I \ | \ \tau \in \tree, \ \Pi_{[1, |\omega|]}(\tau) = \omega, \  \Pi_{[|\omega|+1, |\omega|+n]}(\tau) = \omega^{\prime} \ \} \\
        &= \{\Pi_{[|\omega\omega^{\prime}|+1, |\omega\omega^{\prime}|+1]}(\omega \tau) \in I \ | \ \tau \in \tree, \  \Pi_{[1, |\omega\omega^{\prime}|]}(\tau) = \omega\omega^{\prime} \ \}
        = S(\tree, \omega\omega^{\prime}) \in \mathcal{J}. 
    \end{align*}
    Therefore, we have proved that $\tree_{\omega}$ is a tree with $I$. 
    Note that, by the above argument and the equality $(\tree_{\omega})_{\omega^{\prime}} = \tree_{\omega\omega^{\prime}}$, $S(\tree_{\omega}, \omega^{\prime}\tau) = S(\tree_{\omega\omega^{\prime}}, \tau)$ for each $\omega \in \tree^{[1, *]}$, $\omega^{\prime} \in \tree_{\omega}^{[1, *]}$ and $\tau \in \tree_{\omega\omega^{\prime}}^{[1, *]}$. 
    Obviously, $\tree_{\phi} = \tree$, $\tree_{\phi}^{0} = \tree^{0}$, $\tree^{[m, n]}_{\phi} = \tree^{[m, n]}$ for each $m, n \in \natural$ with $m \leq n$ and $\tree^{[1, *]}_{\phi} = \tree^{[1, *]}$. 
\end{remark}
Note that the papers \cite{BHS1}, \cite{BHS2} and \cite{Sce} introduce the notion of $V$-variability by using the sub-tree. 
Indeed, we say that the tree $\tree$ with $I$ is $V$-variable ($V \in \natural$) if 
\begin{align*}
    \#(\{ \ \tree_{\omega_{[1, n]}} \ | \ \omega_{[1, n]} \in \tree^{[1, n]} \ \}) \leq V
\end{align*}
for each $n \in \natural$. 

We finally prove the following proposition. 
\begin{proposition} \label{tree_characterization}
    Let $I$ be a countable set and $\tree $ be a subset of $I^{\natural}$. 
    Then, $\tree$ is tree with $I$ if and only if $\tree$ is non-empty and compact. 
\end{proposition}
\begin{proof}
    Let $\tree$ be a subset of $I^{\natural}$. 
    Assume that $\tree$ is a tree with $I$. 
    Note that $\#(I_{n}) < \infty$ for each $n \in \natural$ and $\tree \subset \prod_{i = 1}^{\infty} I_{i}$. 
    Indeed, let $\omega = \omega_{1} \omega_{2} \cdots \in \tree$. 
    Then, $\omega_{1} \in \Pi_{1}(\tree) = I_{1}$ and $\omega_{n} \in S(\tree, \omega_{[1, n-1]}) \subset I_{n}$ for each $n \in \natural$ with $n \geq 2$. 
    Therefore, we deduce that $\tree$ is compact.  
    
    It remains to show that if $\tree \subset I^{\natural}$ is non-empty and compact, then $\tree$ is a tree with $I$. 
    Note that $\Pi_{[1, n]}(\tree) \subset I^{n}$ is non-empty and compact for each $n \in \natural$.   
    Now, let $n \in \natural$ and $\omega_{[1, n]} \in \Pi_{[1, n]}(\tree)$ \  ($\omega \in \tree$) and we set  
    \begin{align*}
        S := \{ \tau_{n+1} \in I \ | \ \tau = \tau_{1}\tau_{2} \cdots \in \tree, \ \tau_{[1, n]} = \omega_{[1, n]} \ \} \ ( \neq \emptyset). 
    \end{align*}
    Since the mapping $e \colon S \to \Pi_{[1, n+1]}(\tree)$ defined by
    $e(\tau_{n+1}) := \omega_{[1, n]} \tau_{n+1} \ (\tau_{n+1} \in S)$  is well-defined and injective, we have $(0 < ) \#(S) \leq \#(\Pi_{n+1}(\tree)) (< \infty)$. 
    Thus, we have proved our proposition.  
\end{proof}
\begin{remark} \label{tree_remarks_1variable}
    If $\tree$ is $1$-variable, then the tree $\tree$ has another representation. 
    Indeed, by the definition of $1$-variability of $\tree$, 
    we deduce that $\#(\mathcal{J}_{n}) = 1$ for each $n \in \natural$ and $I_{n} = S(\tree, \omega_{[1, n]})$ for each $\omega_{[1, n]} \in \tree^{[1, n]} \ (n \in \natural)$, and by the induction we deduce that $\tree^{1} = \{\omega_{1} \in I \ | \ \omega \in I_{1} \ \} = I_{1}$ and 
    \begin{align*}
        \tree^{[1, n+1]} 
        &= \Pi_{[1, n+1]}(\tree)
        = \bigcup_{\omega_{[1, n]} \in \tree^{[1, n]}} \{ \omega_{[1, n]} \} \times S(\tree, \omega_{[1, n]})
        = \bigcup_{\omega_{[1, n]} \in \tree^{[1, n]}} \{ \omega_{[1, n]} \} \times I_{n+1} \\
        &= \tree^{[1, n]} \times I_{n+1}
        = I_{1} \times \cdots \times I_{n} \times I_{n+1}
        = \prod_{i = 1}^{n+1} I_{i} 
    \end{align*}
    for each $n \in \natural$. 
    Now, let $x = x_{1}x_{2} \cdots \in \prod_{i = 1}^{\infty} I_{i}$. 
    Since $x \in \tree^{[1, n]} \times \prod_{i = n+1}^{\infty} I_{i}$ for each $n \in \natural$, there exists $\omega^{(n)} = \omega^{(n)}_{1} \omega^{(n)}_{2} \cdots \in \tree$ such that $\omega^{(n)}_{[1, n]} = x_{[1, n]}$ 
    and we deduce that $\omega^{(n)}$ converges to $x$ as $n \in \natural$ tends to infinity. 
    Since $\tree$ is closed, we have $x = \lim_{n \to \infty} \omega^{(n)} \in \tree$. 
    Therefore, we obtain that 
    \begin{align*}
        \tree = \prod_{i = 1}^{\infty} I_{i} \quad \text{and} \quad
        \tree^{[m, n]} = \prod_{i = m}^{n} I_{i} \ (m, n \in \natural \ \text{with} \ m \leq n). 
    \end{align*}
    
    By the similar argument, we also obtain that if $\tree$ is $1$-variable, then   
    \begin{align*}
        \tree_{\omega} = \prod_{i = |\omega|+1}^{\infty} I_{i}, \quad \text{and} \quad
        \tree_{\omega}^{[m, n]} = \prod_{i = |\omega|+m}^{|\omega|+n} I_{i} \ (m, n \in \natural \ \text{with} \ m \leq n). 
    \end{align*}
    for each $\omega \in \tree^{[1, *]}$.  
    Note that this is the case in the paper \cite{ReU}. 
\end{remark}
%
%
%
%
%
%
%
%
%
%
%
%
%
%
%
%
%
%
\subsection{General iterated function systems and the projection maps}
\label{general_IFS_and_proj_map}
In this subsection, we introduce the notion of general IFSs and we consider the projection map for the general IFSs. 
In this paper, general IFS is a pair of a tree and a family of (uniformly) contractive mappings on a complete metric space. 
Later, we introduce the family of the limits for a general IFS by using the projection map for the general IFSs. 

We now introduce the definition of general IFSs. 
\begin{definition} \label{def_of_genral_NIFS}
    Let $(X, \rho)$ be a complete metric space. 
    We say that a pair $(\{ f_{i} \}_{i \in I}, \tree)$ is a general IFS on $(X, \rho)$ with the uniform contraction constant $c \in (0, 1)$ if
    \begin{enumerate}
        \item $\tree$ is a tree with a set $I$ and 
        \item $f_{i} \colon X \to X \ (i \in I)$ is a family of contractive mappings on $X$ with the uniform contraction constant $c$, that is, for all $i \in I$ and $x, y \in X$, 
        \begin{equation*}
            \rho(f_{i}(x), f_{i}(y)) \leq c \ \rho(x, y). 
        \end{equation*}
    \end{enumerate} 
\end{definition}
Note that, for each $i \in I$, there exists the unique fixed point $z_{i}$ of $f_{i}$ since $X$ is complete. 

\begin{lemma} \label{projection_map_conti_lemma}
    Let $(\{ f_{i} \}_{i \in I}, \tree)$ be a general IFS with the uniform contraction constant $c \in (0, 1)$ on a complete metric space $(X, \rho)$ and $z_{i} \in X$ be the unique fixed point of $f_{i} \ (i \in I)$. 
    Suppose that there exists $x \in X$ such that
    \begin{equation*}
            \sum_{n \in \mathbb{N}} \left\{ \max_{i \in I_{n}} \rho(x, z_{i}) \right\} \cdot c^{n} < \infty. 
    \end{equation*}
    Then, there exists $\{ x_{\omega} \}_{\omega \in \tree}$ such that,  $x_{\omega} \in X$ for each $\omega \in \tree$ and 
    \begin{align} \label{projection_map_conti_inequality}
        \sup \left\{ \ \rho(x_{\omega}, x_{\omega^{\prime}}) \ | \ \omega, \omega^{\prime} \in \tree, \omega_{[1, s]} = \omega^{\prime}_{[1, s]} \ \right\}
        \leq 2 \cdot (1+c) \cdot c^{-1} \sum_{k=s+1}^{\infty} \left\{ \max_{i \in I_{k}} \rho(x, z_{i}) \right\} \cdot c^{k}
    \end{align}
    for each $s \in \natural$. 
\end{lemma}
\begin{proof}
    Let $\omega \in \tree$ and we set $g_{m} := f_{\omega_{m}} \ (m \in \natural)$. 
    Note that $\{g_{m}\}_{m \in \natural}$ is a sequence of contractive mappings on $X$ with an uniform contraction constant $c \in (0, 1)$ (see, Definitions \ref{definition_of_nonautonomous_iterations} and \ref{def_of_genral_NIFS}) and satisfies the assumption in Lemma \ref{e_contractiveiteration}. 
    Note that $\omega_{k} \in I_{k}$ for each $k \in \natural$. 
    By the inequality (\ref{basic_rate_of_naifs}) with $m = 1$, we deduce that 
    \begin{align} \label{basic_projection_map_inequality}
        \rho(g_{[1, 1+n]}(x), x_{\omega})
        \leq (1+c) \cdot c^{-1} \sum_{k = n+2}^{\infty} c^{k} \rho(x, z_{\omega_{k}})
        \leq (1+c) \cdot c^{-1} \sum_{k = n+2}^{\infty} \left\{ \max_{i \in I_{k}} \rho(x, z_{i}) \right\} \cdot c^{k}
    \end{align} 
    for all $n \in \naturalz$, where $x_{\omega} \in X$ is the first element $x_{1}$ of the recursively compatible sequence $\{x_{m}\}_{m \in \natural}$ for $\{g_{m}\}_{m \in \natural}$ and $z_{\omega_{k}}$ is the unique fixed point of $g_{k} = f_{\omega_{k}} \ (k \in \natural)$.  
    By the same argument, we also deduce that 
    \begin{align} \label{basic_projection_map_inequality_prime} 
        \rho(g^{\prime}_{[1, 1+n]}(x), x_{\omega^{\prime}})
        \leq (1+c) \cdot c^{-1} \sum_{k = n+2}^{\infty} \left\{ \max_{i \in I_{k}} \rho(x, z_{i}) \right\} \cdot c^{k}
    \end{align}
    for all $n \in \naturalz$ and $\omega^{\prime} \in \tree$, where $g^{\prime}_{m} := f_{\omega^{\prime}_{m}} \ (m \in \natural)$ and $x_{\omega^{\prime}} \in X$ is the first element $x^{\prime}_{1}$ of the recursively compatible sequence $\{x^{\prime}_{m}\}_{m \in \natural}$ for $\{g^{\prime}_{m} \}_{m \in \natural}$.  
     
    Now, let $s \in \natural$ and $\omega, \omega^{\prime} \in \tree$ with $\omega_{[1, s]} = \omega^{\prime}_{[1, s]}$. 
    Note that $g_{m}=g^{\prime}_{m}$ for each $m = 1, \ldots, s$. 
    By the inequality (\ref{basic_projection_map_inequality}) and (\ref{basic_projection_map_inequality_prime}) with $n =s-1 \in \naturalz$, we have 
    \begin{align*}
        \rho(x_{\omega}, x_{\omega^{\prime}}) 
        &\leq \rho(x_{\omega}, g_{[1, s]}(x)) + \rho(g^{\prime}_{[1, s]}(x), x_{\omega^{\prime}}) \notag \\
        &\leq 2 \cdot (1+c) \cdot c^{-1} \sum_{k = s+1}^{\infty} \left\{ \max_{i \in I_{k}} \rho(x, z_{i}) \right\} \cdot c^{k}.  \label{diamestimate}
    \end{align*}
    Thus, we have proved our lemma. 
\end{proof}
We now introduce the definition of the projection map on trees. 
\begin{definition} \label{definition_of_projection_map}
    Let Let $(\{ f_{i} \}_{i \in I}, \tree)$ be a general IFS with the uniform contraction constant $c \in (0, 1)$ on a complete metric space $(X, \rho)$ and $z_{i} \in X$ be the unique fixed point of $f_{i} \ (i \in I)$. 
    Suppose that there exists $x \in X$ such that
    \begin{equation*}
            \sum_{n \in \mathbb{N}} \left\{ \max_{i \in I_{n}} \rho(x, z_{i}) \right\} \cdot c^{n} < \infty. 
    \end{equation*}
    Then, the projection map $\pi \colon \tree \to X$ for $( \{ f_{i} \}_{i \in I}, \{ J_{n} \}_{n \in \natural} )$ is defined by 
    \begin{equation*}
        \pi(\omega) := x_{\omega} \quad ( \omega \in \tree ), 
    \end{equation*} 
    where $x_{\omega} \in X \ (\omega \in \tree)$ are the elements introduced in Lemma \ref{projection_map_conti_lemma} ( the first elements $x_{1}$ of the recursively compatible sequence $\{x_{m}\}_{m \in \natural}$ for $\{f_{\omega_{m}}\}_{m \in \natural}$). 
\end{definition}

\begin{remark} \label{projection_map_subtree}
    By the same argument in Lemma \ref{Cauchyfiniteness}, 
    if a general IFS $(\{f_{i}\}_{i \in I}, \tree)$ with the uniform contraction constant $c \in (0, 1)$ on a complete metric space $(X, \rho)$ satisfies the the condition
    \begin{equation*}
        \sum_{n \in \mathbb{N}} \left\{ \max_{i \in I_{n}} \rho(x, z_{i}) \right\} \cdot c^{n} < \infty. 
    \end{equation*}
    for some $x \in X$ ( where $z_{i} \in X$ be the unique fixed point of $f_{i} \ (i \in I)$), then we have 
    \begin{equation*}
        \sum_{n \in \mathbb{N}} \left\{ \max_{i \in I_{n}} \rho(y, z_{i}) \right\} \cdot c^{n} < \infty. 
    \end{equation*}
    for all $y \in X$. 
    
    In addition, if a general IFS $(\{f_{i}\}_{i \in I}, \tree)$ with the uniform contraction constant $c \in (0, 1)$ on a complete metric space $(X, \rho)$ satisfies the above condition for some $x \in X$, then we also define the projection map for $(\{f_{i}\}_{i \in I}, \tree_{\omega})$ for each $\omega \in \tree^{[1, *]}$. 

    Indeed, Let $\omega \in \tree^{[1, *]}$. 
    By Remark \ref{definition_subtree}, we have 
    $I_{1}^{\omega} := \Pi_{1}(\tree_{\omega}) = S(\tree, \omega) \subset I_{|\omega|+1}$. 
    In addition, let $n \geq 2$. 
    By Remark \ref{definition_subtree}, we deduce that $\{\omega\} \times \tree_{\omega}^{[1, n-1]} \subset \tree^{[1, |\omega|+n-1]}$, $S(\tree_{\omega}, \omega^{\prime}_{[1, n-1]}) = S(\tree, \omega\omega^{\prime}_{[1, n-1]})$ for each $\omega^{\prime}_{[1, n-1]} \in \tree_{\omega}^{[1, n-1]}$ and  
    \begin{align*}
        I_{n}^{\omega} 
        := \bigcup_{\omega^{\prime}_{[1, n-1]} \in \tree_{\omega}^{[1, n-1]}} S(\tree_{\omega}, \omega^{\prime}_{[1, n-1]}) 
        \subset \bigcup_{\tau \in \tree^{[1, |\omega|+n-1]} } S(\tree, \tau) = I_{|\omega|+n}
    \end{align*}
    and it follows that
    \begin{align*}
        \sum_{n \in \mathbb{N}} \left\{ \max_{i \in I_{n}^{\omega}} \rho(x, z_{i}) \right\} \cdot c^{n}
        \leq \sum_{n \in \mathbb{N}} \left\{ \max_{i \in I_{|\omega|+n}} \rho(x, z_{i}) \right\} \cdot c^{n}
        \leq c^{-|\omega|} \sum_{n \in \mathbb{N}} \left\{ \max_{i \in I_{n}} \rho(x, z_{i}) \right\} \cdot c^{n} < \infty. 
    \end{align*}
    Therefore, by Lemma \ref{projection_map_conti_lemma}, the projection map for $( \{ f_{i} \}_{i \in I}, \tree_{\omega} )$ is well-defined for each $\omega \in \tree^{[1, *]}$
    Henceforth, we also denote by $\pi$ the projection map for $(\{ f_{i} \}_{i \in I}, \tree_{\omega} )$. 
\end{remark}
We next define the notions used in Theorem \ref{Lipschitz_of_projection_map}. 
\begin{definition}
    Let $I$ be a set. 
    The shift map $\sigma \colon I^{\natural} \to I^{\natural}$ is defined by 
    \begin{equation*}
        \sigma(\omega) := \omega_{2} \omega_{3} \cdots \quad ( \omega = \omega_{1} \omega_{2} \cdots \in I^{\natural} ).  
    \end{equation*}
    Similarly, the map $\sigma_{i} \colon I^{\natural} \to I^{\natural} \ (i \in I)$ is defined by 
    \begin{equation*}
        \sigma_{i}(\omega) := i \omega \cdots \quad ( \omega \in I^{\natural}). 
    \end{equation*}
    In addition, for each $\omega = \omega_{1} \cdots \omega_{n} \in I^{n} \ (n \in \natural)$, we set $f_{\omega} := f_{\omega_{1}} \circ \cdots \circ f_{\omega_{n}}$, and $f_{\phi} := \textrm{id}_{X}$. 
\end{definition}
We now show the following important theorem in this paper. 
Note that the following theorem is a generalization of the result on \cite{I}. 
\begin{theorem} \label{Lipschitz_of_projection_map}
    Let $\pi \colon \tree \to X$ be the projection map for a general IFS $(\{ f_{i} \}_{i \in I}, \tree)$ with the uniform contraction constant $c \in (0, 1)$ on a complete metric space $(X, \rho)$ such that
    \begin{equation*}
        \sum_{n \in \mathbb{N}} \left\{ \max_{i \in I_{n}} \rho(x, z_{i}) \right\} \cdot c^{n} < \infty. 
    \end{equation*}
    for some $x \in X$, where $z_{i} \in X$ be the unique fixed point of $f_{i} \ (i \in I)$. 
    Then, $\pi \colon \tree \to X$ is uniformly continuous on $\tree$. 
    In addition, for each $\omega \in \tree^{[1, *]}$, $f_{\omega^{\prime}} \circ \pi = \pi \circ \sigma_{\omega^{\prime}}$ on $\tree_{\omega\omega^{\prime}}$ for each $\omega^{\prime} \in \tree_{\omega}^{[1, *]}$. 
\end{theorem}
\begin{proof}
    We first show that $\pi \colon \tree \to X$ is uniformly continuous on $\tree$. 
    Let $\epsilon > 0$. 
    Then, there exists $M \in \mathbb{N}$ such that 
    $2 \cdot (1+c) \cdot c^{-1} \sum_{k=M+1}^{\infty} \left\{ \max_{i \in I_{k}} \rho(x, z_{i}) \right\} \cdot c^{k} < \epsilon$. 
    We set $U := \{ \omega_{[1, M]} \} \times \prod_{i = M+1}^{\infty} I \subset I^{\natural} \ (\omega_{[1, M]} \in \tree^{[1, M]})$ and let $\tau, \tau^{\prime} \in U$.
    Since $\tau_{i} = \tau^{\prime}_{i}$ for each $i = 1, \ldots M$ and by Lemma \ref{projection_map_conti_lemma} we deduce that 
    \begin{equation}
        d(\pi(\tau), \pi(\tau^{\prime})) \leq 2 \cdot (1+c) \cdot c^{-1} \sum_{k=M+1}^{\infty} \left\{ \max_{i \in I_{k}} \rho(x, z_{i}) \right\} \cdot c^{k} < \epsilon.   
    \end{equation}
    Therefore, we have proved that $\pi \colon \tree \to X$ is uniformly continuous on $\tree$. 
    We now show that, for each $\omega \in \tree^{[1, *]}$, $f_{\omega^{\prime}} \circ \pi = \pi \circ \sigma_{\omega^{\prime}}$ on $\tree_{\omega\omega^{\prime}}$ for each $\omega^{\prime} \in \tree_{\omega}^{[1, *]}$. 
    Let $\omega \in \tree^{[1, *]}$, $\omega^{\prime} \in \tree_{\omega}^{[1, *]}$ and $\tau \in \tree_{\omega\omega^{\prime}}$. 
    Then, we have $\sigma_{\omega^{\prime}}(\tau) \in \tree_{\omega}$. 
    Indeed, by the definition of $\tree_{\omega\omega^{\prime}}$, we have $\omega\omega^{\prime}\tau \in \tree$ and $\pi_{(|\omega|, \infty)}(\omega\omega^{\prime}\tau) = \omega^{\prime}\tau$. 
    It follows that 
    \begin{equation*}
        \sigma_{\omega^{\prime}}(\tau) 
        = \omega^{\prime} \tau 
        = \pi_{(|\omega|, \infty)}(\omega\omega^{\prime}\tau) \in \tree_{\omega}. 
    \end{equation*}
    
    Now, recall that, by the condition in Theorem \ref{Lipschitz_of_projection_map}, $\{f_{\tau_{m}}\}_{m \in \natural}$ is a sequence of contractive mappings on $X$ with an uniform contraction constant $c \in (0, 1)$ and satisfies the condition in Lemma \ref{e_contractiveiteration}. 
    By definition of $\pi$ for general IFS $(\{ f_{i} \}_{i \in I}, \tree_{\omega\omega^{\prime}})$ and the inequality (\ref{basic_rate_of_naifs}) with $m = 1$ and $y = x$, we have
    \begin{align*}
        &\rho(f_{\tau_{[1, n+1]}}(x), \pi(\tau) )
        \leq (1+c) \cdot c^{-1} \sum_{k = n+2}^{\infty} c^{k} \rho(x, z_{\tau_{k}}) \\
        &\leq (1+c) \cdot c^{-1} \sum_{k = n+2}^{\infty} \left\{ \max_{i \in I^{\omega\omega^{\prime}}_{k}} \rho(x, z_{i}) \right\} \cdot c^{k}
        \leq (1+c) \cdot c^{-(|\omega\omega^{\prime}|+1)} \sum_{k = n+2}^{\infty} \left\{ \max_{i \in I_{k}} \rho(x, z_{i}) \right\} \cdot c^{k}
    \end{align*}
    for all $\tau \in \tree_{\omega\omega^{\prime}}$ and $n \in \naturalz$, where $z_{\tau_{k}}$ is the unique fixed point of $f_{\tau_{k}} \ (k \in \natural)$ and we use the fact that 
    $\tau_{k} \in I^{\omega\omega^{\prime}}_{k} \subset I_{|\omega\omega^{\prime}|+k}$ for each $k \in \natural$ in Remark \ref{projection_map_subtree}. 
    Thus, by the definition of $\pi(\tau)$, we have
    \begin{align*}
        &\rho(f_{(\omega^{\prime}\tau)_{[1, n+|\omega^{\prime}|+1]}}(x), f_{\omega^{\prime}}(\pi(\tau))) 
        \leq c^{|\omega^{\prime}|} \cdot \rho(f_{\tau_{[1, n+1]}}(x), \pi(\tau) ) \\
        &\leq (1+c) \cdot c^{-(|\omega|+1)} \cdot \sum_{k = n+2}^{\infty} \left\{ \max_{i \in I_{k}} \rho(x, z_{i}) \right\} \cdot c^{k}
    \end{align*}
    for each $n \in \natural$ and it follows that $f_{(\omega^{\prime}\tau)_{[1, n+|\omega^{\prime}|+1]}}(x)$ converges to $f_{\omega^{\prime}}(\pi(\tau))$ as $n$ tends to infinity. 
    On the other hand, by the definition of $\pi(\omega^{\prime}\tau)$, $f_{(\omega^{\prime}\tau)_{[1, n+|\omega^{\prime}|+1]}}(x)$ converges to $\pi(\omega^{\prime}\tau)$ as $n$ tends to infinity and it follows that $f_{\omega^{\prime}}(\pi(\tau)) = \pi(\omega^{\prime}\tau) = \pi(\sigma_{\omega^{\prime}}(\tau))$ for each $\tau \in \tree_{\omega\omega^{\prime}}$. 
    Hence, we have proved our lemma. 
\end{proof}
\begin{remark} \label{choice_of_sequence}
    In Theorem \ref{Lipschitz_of_projection_map}, if there exist $C > 0$ and a non-negative-valued sequence $b$ such that $\left\{ \max_{i \in I_{l}} \rho(x, z_{i}) \right\} \cdot c^{l} \leq C \cdot b(l)$ with $\sum_{l \in \natural} b(l) < \infty$, then the similar result as the Lemma \ref{projection_map_conti_lemma} and Theorem \ref{Lipschitz_of_projection_map} also holds. 
\end{remark}

%
%
%
%
%
%
%
%
%
%
%
%
%
%
%
%
%
%
%
%
%
%
%
%
%
%
%
%
%
%
%
%
%
%
%
%
%
%
%
%
%
%
%
%
%
%
%
%
%
%
%
%
%
%
%

\section{The family of limit sets for general IFSs} \label{limit_set}
In this section, we now introduce the definition of the family of the limit set for general IFSs and show the uniqueness (in some sense) of the family of the limit sets for the general IFS. 
Note that the main result (Theorem \ref{main}) is a generalization of the result on \cite{I}. 
\begin{definition} \label{Definition_of_limit_set}
        Let $(\{ f_{i} \}_{i \in I}, \tree)$ be a general IFS with the uniform contraction constant $c \in (0, 1)$ on a complete metric space $(X, \rho)$ and $z_{i} \in X$ be the unique fixed point of $f_{i} \ (i \in I)$. 
    Suppose that there exists $x \in X$ such that
    \begin{equation*}
            \sum_{n \in \mathbb{N}} \left\{ \max_{i \in I_{n}} \rho(x, z_{i}) \right\} \cdot c^{n} < \infty. 
    \end{equation*}
    Then, the non-empty compact subset $\pi(\tree)$ of $X$ is called the limit set for $(\{ f_{i} \}_{i \in I}, \tree)$, where continuous map $\pi \colon \tree \to X$ is the projection map for $(\{ f_{i} \}_{i \in I}, \tree)$ introduced in Definition \ref{definition_of_projection_map} (also, see Proposition \ref{tree_characterization} and Theorem \ref{Lipschitz_of_projection_map} ). 
\end{definition}
\begin{remark}
    By Remark \ref{projection_map_subtree}, if a general IFS $(\{f_{i}\}_{i \in I}, \tree)$ with the uniform contraction constant $c \in (0, 1)$ on a complete metric space $(X, \rho)$ satisfies the condition
    \begin{equation*}
        \sum_{n \in \mathbb{N}} \left\{ \max_{i \in I_{n}} \rho(x, z_{i}) \right\} \cdot c^{n} < \infty. 
    \end{equation*}
    for some $x \in X$ ( where $z_{i} \in X$ is the unique fixed point of $f_{i} \ (i \in I)$), then we also define the limit set $\pi(\tree_{\omega})$ for the general IFS $(\{f_{i}\}_{i \in I}, \tree_{\omega})$ for each $\omega \in \tree^{[1, *]}$. 
    Note that $\pi(\tree_{\omega})$ is also non-empty compact subset of $X$. 
\end{remark}

To discuss the uniqueness of the limit set for general IFSs, we introduce the notion of the Hausdorff distance and its properties. 
Let $\mathcal{K}(X)$ be the set of non-empty compact subsets in a complete metric space $(X, \rho)$. 
For each $\epsilon > 0$ and set $A \subset X$, we set $A_{\epsilon} := \{ x \in X \ | \ {}^{\exists}a \in A, \ \text{s.t.} \  \rho(a, x) \leq \epsilon \}$.  
Let $\rho_{H}$ be the Hausdorff distance on $\mathcal{K}(X)$ defined by 
\begin{equation*}
    \rho_{H}(A, B) := \inf \{ \epsilon > 0 \ | \ A \subset B_{\epsilon}, B \subset A_{\epsilon} \} \quad (A, B \in \mathcal{K}(X)). 
\end{equation*}
Note that $\rho_{H}(\{a\}, \{a^{\prime}\}) = \rho(a, a^{\prime})$ for each $a, a^{\prime} \in X$ and $\rho_{H}(A, \{a^{\prime}\}) = \sup_{a \in A} \rho(a, a^{\prime})$ for each $A \in \mathcal{K}(X)$ and $a^{\prime} \in X$. 
Also, note that since $(X, \rho)$ is complete, $(\mathcal{K}(X), \rho_{H})$ is also complete (For example, see \cite{Ki}). 
Indeed, if $\{L_{n}\}_{n \in \natural}$ is a Cauchy sequence in $(\mathcal{K}(X), \rho_{H})$, then $\{L_{n}\}_{n \in \natural}$ converges to 
\begin{equation} \label{KX_Cauchy_limit}
    L := \bigcap_{n \in \natural} \overline{\bigcup_{ k \in \natural, k \geq n} L_{k}}^{\rho} \in \mathcal{K}(X)
\end{equation}
as $n$ tends to infinity, where $\overline{A}^{\rho}$ is the closure of $A \subset X$ with respect to the metric $\rho$. 

%
%
%
%
%
%
%
%
%
%
%
%
%
%
%
%
%
%
%
%
%
Before we prove the main theorem, we show the characterization of the projection map for a general IFSs $(\{f_{i}\}_{i \in I}, \tree)$. 
\begin{proposition} \label{justification_projection_map}
    Let $(\{ f_{i} \}_{i \in I}, \tree)$ be a general IFS with the uniform contraction constant $c \in (0, 1)$ on a complete metric space $(X, \rho)$ and $z_{i} \in X$ be the unique fixed point of $f_{i} \ (i \in I)$. 
    Suppose that there exists $x \in X$ such that
    \begin{equation*}
            \sum_{n \in \mathbb{N}} \left\{ \max_{i \in I_{n}} \rho(x, z_{i}) \right\} \cdot c^{n} < \infty. 
    \end{equation*}
    Then, for each $\omega \in \tree$ and $A \in \mathcal{K}(X)$, the sequence $\{ f_{\omega_{[1, n]}}(A)\}_{n \in \natural}$ of compact subsets converges to the single set $\{ \pi(\omega) \}$ as $n$ tends to infinity in sense of the Hausdorff distance. 
\end{proposition}
\begin{proof}
    Let $\omega \in \tree$ and we set $g_{m} := f_{\omega_{m}} \ (m \in \natural)$.  
    Note that $\{g_{m}\}_{m \in \natural}$ is a sequence of contractive mappings on $X$ with an uniform contraction constant $c \in (0, 1)$ (see, Definitions \ref{definition_of_nonautonomous_iterations} and \ref{def_of_genral_NIFS}) and satisfies the condition in Lemma \ref{e_contractiveiteration}. 
    Also, note that $\omega_{k} \in I_{k}$ for each $k \in \natural$. 
    By the inequality (\ref{basic_rate_of_naifs}) with $m = 1$, we deduce that 
    \begin{align*}
        &\rho_{H}(f_{\omega_{[1, n]}}(\{x\}), \{\pi(\omega)\})
        =\rho_{H}(\{f_{\omega_{[1, n]}}(x)\}, \{\pi(\omega)\})
        =\rho(f_{\omega_{[1, n]}}(x), \pi(\omega)) \\
        &\leq (1+c) \cdot c^{-1} \sum_{k = n+1}^{\infty} c^{k} \rho(x, z_{\omega_{k}})
        \leq (1+c) \cdot c^{-1} \sum_{k = n+1}^{\infty} \left\{ \max_{i \in I_{n}} \rho(x, z_{i}) \right\} \cdot c^{n}
    \end{align*}
    for all $n \in \natural$, where $z_{\omega_{k}}$ is the unique fixed point of $g_{k} = f_{\omega_{k}} \ (k \in \natural)$ and we use the fact $\pi(\omega)$ is the first element $x_{1}$ of the recursively compatible sequence $\{x_{m}\}_{m \in \natural}$ for $\{g_{m}\}_{m \in \natural}$.  
    In addition, for each $\omega \in \tree$, $A \in \mathcal{K}(X)$ and $n \in \natural$, we have
    \begin{align*}
        \rho_{H}(f_{\omega_{[1, n]}}(\{x\}), f_{\omega_{[1, n]}}(A))
        \leq c^{n} \cdot \rho_{H}(\{x\}, A). 
    \end{align*}
    Therefore, we have proved our proposition. 
\end{proof}
\begin{remark} \label{projection_map_connection_remark}
    Proposition \ref{justification_projection_map} shows the reason why we call $\pi$ the projection map, and that if we only consider the construction of the projection map for general IFSs then we do not assume the boundedness or compactness of $(X, \rho)$ and we can choice any non-empty compact subset as the initial point for each iteration. 
    
    Indeed, the projection map (or the coding map) for IFSs is ``usually" defined by the unique element of the intersection of non-increasing compact subsets generated by contractive mappings.
    For example, in Rempe-Gillen's and Urba\'{n}ski's paper \cite{ReU}, the projection map on compact metric space $(X, \rho)$ is defined by the intersection of non-increasing compact subsets $\{ f_{\omega_{[1, m]}}(X) \}_{m \in \natural}$ generated by sequence of contractive mappings $\{f_{\omega_{m}}\}_{m \in \natural}$. 
    Note that, if a non-increasing sequence $\{A_{m}\}_{m \in \natural}$ of non-empty compact subsets satisfies the condition $\diam_{\rho}(A_{m}) \to 0$ as $m$ tends to infinity, then we deduce that the intersection $ \cap_{m \in \natural} A_{m}$ is a single set and is the limit point of the sequence $\{ A_{m} \}_{m \in \natural}$ in sense of the Hausdorff distance. 
\end{remark}
We now prove the main theorem in this paper. 
\begin{theorem} \label{main}
    Let $(\{ f_{i} \}_{i \in I}, \tree)$ be a general IFS with the uniform contraction constant $c \in (0, 1)$ on a complete metric space $(X, \rho)$ and $z_{i} \in X$ be the unique fixed point of $f_{i} \ (i \in I)$. 
    Suppose that there exists $x \in X$ such that
    \begin{equation*}
            \sum_{n \in \mathbb{N}} \left\{ \max_{i \in I_{n}} \rho(x, z_{i}) \right\} \cdot c^{n} < \infty. 
    \end{equation*}
    Then, we have the following properties: 
    \begin{enumerate}
        \item $\displaystyle \bigcup_{\omega^{\prime}_{[1, n]} \in \tree_{\omega}^{[1, n]}} f_{\omega^{\prime}_{[1, n]}} (\pi(\tree_{\omega \omega^{\prime}})) = \pi(\tree_{\omega})$ for each $\omega \in \tree^{[1, *]}$ \quad and  
        \item for each $A \in \mathcal{K}(X)$, there exists $D^{\prime}(A) > 0$ such that for each $\omega \in \tree^{[1, *]}$
        \begin{align*}
            \rho_{H}( A, \pi(\tree_{\omega}))
            \leq D^{\prime}(A) \cdot \max \left\{ 1,  c^{-(|\omega|+1)} 
            b_{x}(|\omega|+1) \right\}, 
        \end{align*}
    \end{enumerate}
    where $b_{x}(l) := \sum_{k \geq l} \{ \max_{i \in I_{k}}\rho(x, z_{i}) \} \cdot c^{k} \ (l \in \natural)$. 
    In addition, the family of compact subsets $\{K_{\omega}\}_{\omega \in \tree}$ of $X$ with the above properties (i) and (ii) is unique. 
    Moreover, for each $A \in \mathcal{K}(X)$, there exists $D(A) >0$ such that for each $\omega \in \tree^{[1, *]}$ and $n \in \natural$, 
    \begin{align}\label{main_conv_rate}
        \rho_{H}\left( \bigcup_{\omega^{\prime}_{[1, n]} \in \tree_{\omega}^{[1, n]} } f_{\omega^{\prime}_{[1, n]}} (A), \pi(\tree_{\omega}) \right) 
        \leq D(A) \cdot c^{-(|\omega|+1)} \cdot \max \{ c^{n+|\omega|+1}, b_{x}(n+|\omega|+1) \}.  
    \end{align}
\end{theorem}
\begin{remark}
    Note that it is necessary to assume that $\sum_{n \in \mathbb{N}} \left\{ \max_{i \in I_{n}} \rho(x, z_{i}) \right\} \cdot c^{n} < \infty$ for some $x \in X$. 
    Indeed, if the tree is a single set then the general IFSs is reduced to the non-autonomous iteration (see, Example \ref{1dimex}). 
    
    Under the condition, we obtain the inequality (\ref{main_conv_rate}) which shows that for all $\omega \in \tree^{[1, *]}$, $\{ \bigcup_{\omega^{\prime}_{[1, n]} \in \tree_{\omega}^{[1, n]} } f_{\omega^{\prime}_{[1, n]}}(A) \}_{n \in  \natural}$ subsets converges to $\pi(\tree_{\omega})$ as $n$ tends to infinity in sense of the Hausdorff distance. 
    In addition, the inequality (\ref{main_conv_rate}) shows that the limit point does not depend on the starting point $A \in \mathcal{K}(X)$ and the starting point depends on only the constant of the convergence rate (note that the base point $x \in X$ depends on the convergence rate). 
    However, the convergence is not always a exponentially fast rate. 
    Indeed, if the tree is a single set then the general IFSs is reduced to the non-autonomous iteration (see, Example \ref{polyfastconvexample}). 
\end{remark}
\begin{proof}[proof of Theorem \ref{main}]
    We first show that the property (i) in Theorem \ref{main}. 
    Indeed, let $\omega \in \tree^{[1, *]}$ and $n \in \natural$. 
    Note that by the definition of sub-trees (see, Remark \ref{definition_subtree}), $\omega \omega^{\prime}_{[1, n]} \in \tree^{[1, |\omega| + n]}$ and the set $\tree_{\omega \omega^{\prime}_{[1, n]}}$ is well-defined for each for each $\omega^{\prime}_{[1, n]} \in \tree_{\omega}^{[1, n]}$, and we deduce that  $\bigcup_{\omega^{\prime}_{[1, n]} \in \tree_{\omega}^{[1, n]}} \{ \omega^{\prime}_{[1, n]} \} \times \tree_{\omega \omega^{\prime}_{[1, n]}} = \tree_{\omega}$. 
    Then, by Theorem \ref{Lipschitz_of_projection_map}, 
    \begin{align*}
        \bigcup_{\omega^{\prime}_{[1, n]} \in \tree_{\omega}^{[1, n]}} f_{\omega^{\prime}_{[1, n]}} (\pi(\tree_{\omega \omega^{\prime}}))
        =\bigcup_{\omega^{\prime} \in \tree_{\omega}^{[1, n]}} \pi(\sigma_{\omega^{\prime}_{[1, n]}}(\tree_{\omega \omega^{\prime}}))
        = \pi\left( \bigcup_{\omega^{\prime} \in \tree_{\omega}^{[1, n]}} \{ \omega^{\prime} \} \times \tree_{\omega \omega^{\prime}} \right)
        =\pi(\tree_{\omega}). 
    \end{align*}
    Therefore, we have proved the property (i) in Theorem \ref{main}. 
    Note that 
    $f_{\omega^{\prime}}(\pi(\tree_{\omega})) \in \mathcal{K}(X)$ since $f_{\omega^{\prime}}$ is continuous on $X$ for each $\omega^{\prime} \in \tree_{\omega}^{[1, *]}$. 
    
    We next show that for each $n \in \natural$, $\omega \in \tree^{[1, *]}$ and $A \in \mathcal{K}(X)$, 
    \begin{align} 
        &\rho_{H}\left( \bigcup_{\omega^{\prime}_{[1, n]} \in \tree_{\omega}^{[1, n]} } f_{\omega^{\prime}_{[1, n]}} ( \{ x \} ), \bigcup_{\omega^{\prime}_{[1, n]} \in \tree_{\omega}^{[1, n]} }f_{\omega^{\prime}_{[1, n]}} (A) \right) 
        \leq c^{n} \cdot \sup_{a \in A} \rho(x, a) \quad \text{and} \quad \label{basic_inequality1_in_main_theorem} \\
        &\rho_{H}\left( \bigcup_{\omega^{\prime}_{[1, n]} \in \tree_{\omega}^{[1, n]} } f_{\omega^{\prime}_{[1, n]}} ( \{ x \} ), \pi(\tree_{\omega}) \right)
        \leq (1+c) \cdot c^{-(|\omega|+1)} \sum_{k = |\omega|+n+1}^{\infty} c^{k} \max_{i \in I_{k}} \rho(x, z_{i}). \label{basic_inequality2_in_main_theorem}
    \end{align}
    Indeed, by properties of the Hausdorff distance, we have
    \begin{align*}
        &\rho_{H}\left( \bigcup_{\omega^{\prime}_{[1, n]} \in \tree_{\omega}^{[1, n]} } f_{\omega^{\prime}_{[1, n]}} ( \{ x \} ), \bigcup_{\omega^{\prime}_{[1, n]} \in \tree_{\omega}^{[1, n]} }f_{\omega^{\prime}_{[1, n]}} (A) \right)
        \leq \max_{\omega^{\prime}_{[1, n]} \in \tree^{[1, n]}} \rho_{H}\left( f_{\omega^{\prime}_{[1, n]}} (\{x\}), f_{\omega^{\prime}_{[1, n]}} (A) \right) \\ 
        &\leq \max_{\omega^{\prime}_{[1, n]} \in \tree_{\omega}^{[1, n]}} c^{n} \cdot \rho_{H}(\{x\}, A)
        \leq c^{n} \cdot \rho_{H}(\{x\}, A)
        \leq c^{n} \cdot \sup_{a \in A} \rho(x, a). 
    \end{align*}
    In addition, by the property (i) in Theorem \ref{main} and the above argument, we have 
    \begin{align*}
        &\rho_{H}\left( \bigcup_{\omega^{\prime}_{[1, n]} \in \tree_{\omega}^{[1, n]} } f_{\omega^{\prime}_{[1, n]}} ( \{ x \} ), \pi(\tree_{\omega}) \right)
        \leq 
        c^{n} \cdot \sup_{\tau \in \tree_{\omega\omega^{\prime}_{[1, n]}}} \rho \left( x, \pi(\tau) \right). 
    \end{align*}
    On the other hand, recall that $\omega\omega^{\prime}_{[1, n]}\tau \in \tree$ and the first element of the recursively compatible sequence for $\{f_{\tau_{m}}\}_{m \in \natural}$ is the $(|\omega|+n+1)$-th element of the recursively compatible sequence for $\{f_{(\omega\omega^{\prime}_{[1, n]}\tau)_{m}}\}_{m \in \natural}$.  
    By the property (ii) in Lemma \ref{e_contractiveiteration} (or the inequality (\ref{asym_bdd_of_x_m})) with $y  = x$ and $m = |\omega|+n+1$, it follows that
    \begin{align*}
        c^{n} \cdot \sup_{\tau \in \tree_{\omega\omega^{\prime}_{[1, n]}}} \rho \left( x, \pi(\tau) \right)
        &\leq c^{n} \cdot (1+c) \cdot c^{-(|\omega|+n+1)} \sum_{k = |\omega|+n+1}^{\infty} c^{k} \rho(x, z_{(\omega\omega^{\prime}_{[1, n]}\tau)_{k}}) \\
        &\leq (1+c) \cdot c^{-(|\omega|+1)} \sum_{k = |\omega|+n+1}^{\infty} c^{k} \max_{i \in I_{k}} \rho(x, z_{i}). 
    \end{align*}
    Therefore, we have proved the inequalities  (\ref{basic_inequality1_in_main_theorem}) and (\ref{basic_inequality2_in_main_theorem}). 
    
    Now, we show the property (ii) in Theorem \ref{main} and the uniqueness of the family of the compact sets with the properties (i) and (ii) in Theorem \ref{main}. 
    
    To show the property (ii) in Theorem \ref{main}, let $A \in \mathcal{K}(X)$ and $\omega \in \tree^{[1, *]}$, and we set 
    $D^{\prime}(A) := \max \{ (1+c), \sup_{a \in A} \rho(x,a) \} \ ( > 0)$.  
    Note that $\omega^{\prime}_{1} \in I_{|\omega|+1}$ for each $\omega^{\prime}_{1} \in \tree_{\omega}^{1}$ and by Lemma \ref{collage} we have 
    \begin{align} 
        &\rho_{H}\left( \{x\}, \bigcup_{\omega^{\prime}_{1} \in \tree_{\omega}^{1}} f_{\omega^{\prime}_{1}}(\{x\})\right)
        \leq \max_{\omega^{\prime}_{1} \in \tree_{\omega}^{1}} \rho_{H}\left( \{x\}, f_{\omega^{\prime}_{1}}(\{x\})\right)
        = \max_{\omega^{\prime}_{1} \in \tree_{\omega}^{1}} \rho_{H}\left( \{x\}, \{ f_{\omega^{\prime}_{1}}(x) \} \right) \notag \\
        &\quad = \max_{\omega^{\prime}_{1} \in \tree_{\omega}^{1}} \rho\left( x, f_{\omega^{\prime}_{1}}(x) \right)
        = \max_{\omega^{\prime}_{1} \in \tree_{\omega}^{1}} (1+c) \cdot \rho\left( x, z_{\omega^{\prime}_{1}} \right)
        \leq (1+c) \cdot \max_{i \in I_{|\omega|+1}} \rho\left( x, z_{i} \right). \label{property2_lemma_main_theorem} 
    \end{align}
    By the inequality (\ref{basic_inequality1_in_main_theorem}) with $n = 1$ and the above inequality, we have
    \begin{align*}
        \rho_{H}( \{ x \}, \pi(\tree_{\omega}))
        &\leq \rho_{H} \left( \{ x \},  \bigcup_{\omega^{\prime}_{1} \in \tree_{\omega}^{1}} f_{\omega^{\prime}_{1}}(\{x\}) \right)
        +\rho_{H}\left( \bigcup_{\omega^{\prime}_{1} \in \tree_{\omega}^{1} } f_{\omega^{\prime}_{1}} ( \{ x \} ), \pi(\tree_{\omega}) \right) \\
        &\leq (1+c) \cdot \max_{i \in I_{|\omega|+1}} \rho\left( x, z_{i} \right) + (1+c) \cdot c^{-(|\omega|+1)} \sum_{k = |\omega|+2}^{\infty} c^{k} \max_{i \in I_{k}} \rho(x, z_{i}) \\
        &= (1+c) \cdot c^{-(|\omega|+1)} b_{x}(|\omega|+1). 
    \end{align*}
    It follows that
    \begin{align*}
        &\rho_{H}( A, \pi(\tree_{\omega}))
        \leq \rho_{H}( A, \{x\}) + \rho(\{x\}, \pi(\tree_{\omega})) 
        \leq D^{\prime}(A) \cdot \max \left\{ 1,  c^{-(|\omega|+1)} b_{x}(|\omega|+1) \right\}. 
    \end{align*}
    
    To show the uniqueness of the family of the compact sets with the properties (i) and (ii) in Theorem \ref{main}, let $\{K_{\omega}\}_{\omega \in \tree^{[1, *]}}$ be a family of the compact sets with the properties. 
    Also, let $\omega \in \tree^{[1, *]}$ and $n \in \natural$. 
    Then, by the property (i) in Theorem \ref{main} and properties of the Hausdorff distance, we have
    \begin{align*}
        \rho_{H}(\pi(\tree_{\omega}), K_{\omega})
        &\leq \rho_{H}\left(\bigcup_{\omega^{\prime}_{[1, n]} \in \tree_{\omega}^{[1, n]} } f_{\omega^{\prime}_{[1, n]}}(\pi(\tree_{\omega\omega^{\prime}_{[1, n]}})), \bigcup_{\omega^{\prime}_{[1, n]} \in \tree_{\omega}^{[1, n]} } f_{\omega^{\prime}_{[1, n]}}(\{x\})\right) \\
        &\qquad + \rho_{H}\left(\bigcup_{\omega^{\prime}_{[1, n]} \in \tree_{\omega}^{[1, n]} } f_{\omega^{\prime}_{[1, n]}}(\{x\}), \bigcup_{\omega^{\prime}_{[1, n]} \in \tree_{\omega}^{[1, n]} } f_{\omega^{\prime}_{[1, n]}}(K_{\omega\omega^{\prime}_{[1, n]}}) \right) \\
        &\leq c^{n} \rho_{H}\left( \pi( \tree_{\omega\omega^{\prime}_{[1, n]}}), \{x\} \right) + c^{n} \rho_{H}( \{x\} , K_{\omega\omega^{\prime}_{[1, n]}} ). 
    \end{align*}
    Moreover, by the property (ii), we deduce that 
    \begin{align*}
        &c^{n} \rho_{H} \left( \{x\}, \pi(\tree_{\omega\omega^{\prime}_{[1, n]}}) \right)
        \leq D^{\prime}(\{x\}) \cdot c^{n} \cdot \max \left\{ 1,  c^{-(|\omega\omega^{\prime}_{[1, n]}|+1)} 
        b_{x}(|\omega\omega^{\prime}_{[1, n]}|+1)
        \right\} \\
        &\qquad = D^{\prime}(\{x\}) \cdot c^{-(|\omega|+1)} \cdot \max \left\{ c^{|\omega|+n+1}, 
        b_{x}(|\omega|+n+1)
        \right\} \to 0 
    \end{align*}
    as $n$ tends to infinity, and by the same argument $c^{n} \rho_{H}( \{x\}, K_{\omega\omega^{\prime}_{[1, n]}}) \to 0$ as $n$ tends to infinity. 
    It follows that $\pi(\tree_{\omega}) = K_{\omega}$ for each $\omega \in \tree^{[1, *]}$. 
    Thus, we have proved the property (ii) in Theorem \ref{main} and the uniqueness of the family of the compact sets with the properties (i) and (ii) in Theorem \ref{main}. 
    
    We finally show that the inequality (\ref{main_conv_rate}). 
    To show this, let $A \in \mathcal{K}(X)$, $\omega \in \tree^{[1, *]}$ and $n \in \natural$. 
    We set $D(A) := \max\{ (1+c), \sup_{a \in A} \rho(x, a) \} \ (> 0)$. 
    Then, by the inequalities (\ref{basic_inequality1_in_main_theorem}) and  (\ref{basic_inequality2_in_main_theorem}), we have
    \begin{align*}
        &\rho_{H}\left( \bigcup_{\omega^{\prime}_{[1, n]} \in \tree_{\omega}^{[1, n]} } f_{\omega^{\prime}_{[1, n]}} (A), \pi(\tree_{\omega}) \right) \\
        &\leq \rho_{H}\left( \bigcup_{\omega^{\prime}_{[1, n]} \in \tree_{\omega}^{[1, n]} } \! \! \! f_{\omega^{\prime}_{[1, n]}} (A), \bigcup_{\omega^{\prime}_{[1, n]} \in \tree_{\omega}^{[1, n]} } \! \! \! f_{\omega^{\prime}_{[1, n]}} (\{ x \}) \right) + \rho_{H}\left( \bigcup_{\omega^{\prime}_{[1, n]} \in \tree_{\omega}^{[1, n]} } \! \! \! f_{\omega^{\prime}_{[1, n]}} ( \{ x \} ), \pi(\tree_{\omega}) \right)\\
        &\leq c^{n} \cdot \sup_{a \in A} \rho(x, a) + (1+c) \cdot c^{-(|\omega|+1)} \sum_{k = |\omega|+n+1}^{\infty} c^{k} \max_{i \in I_{k}} \rho(x, z_{i}) \\
        &\leq D(A) \cdot c^{-(|\omega|+1)} \cdot \max \{ c^{n+|\omega|+1}, b_{x}(n+|\omega|+1) \}.  
    \end{align*}
    Hence, we have proved our theorem. 
\end{proof}
\begin{remark} \label{limit_set_boundedness_remark}
    In the paper \cite{BHS1}, the limit set is defined by the limit point of the sequence $\{ \bigcup_{\omega^{\prime}_{[1, n]} \in \tree_{\omega}^{[1, n]} } f_{\omega^{\prime}_{[1, n]}}(A) \}_{n \in  \naturalz}$ of non-empty compact subsets with the uniform contractivity condition and the following condition: $\sup_{i \in I} \rho(f_{i}(x), x) < \infty$ for some $x \in X$.  
        
    However, Theorem \ref{main} shows that if we only consider the existence and uniqueness family of the limit sets for general IFSs, we obtain the existence and uniqueness of family of the limit sets under the weaker conditions than under the above condition. 
    Indeed, by Lemma \ref{collage}, 
    \begin{align*}
        \sup_{i \in I} \rho(x, z_{i})
        \leq \frac{1}{(1-c)} \sup_{i \in I} \rho(f_{i}(x), x) < \infty 
    \end{align*}
    and it follows that $\sum_{n \in \mathbb{N}} \left\{ \max_{i \in I_{n}} \rho(x, z_{i}) \right\} \cdot c^{n} < \infty$. 
    
    In papers \cite{Mo}, \cite{HZ}, \cite{GM} and \cite{HRWW}, the limit set is generated by the compatible compact subsets (that is, the limit set (the Moran set) is generated by the basic sets with the Moran structure).  
    However, the family of limit sets $\{ \pi(\tree) \}_{\omega \in \tree^{[1, *]}}$ for general IFSs in this paper is compatible with definition of the Moran structure $\{ \pi(\tree_{\omega}) \}_{ \omega \in \tree^{[1, *]} }$ and equals the Moran sets. 
    Indeed, by Lemma \ref{projection_map_conti_lemma}, $\pi(\tree_{\omega})$ is compact for each $\omega \in \tree^{[1, *]}$ and $\diam_{\rho}(\pi(\tree_{\omega}))$ converges $0$ as $|\omega|$ tends to $\infty$ uniformly with respect to $\omega \in \tree^{[1, *]}$. 
    In addition, by Theorem \ref{main}, we have $f_{i}(\pi(\tree_{\omega i})) \subset \pi(\tree_{\omega})$ and $f_{i}(\pi(\tree_{\omega})) = \pi(\tree_{\omega i})$ for each $\omega \in \tree^{[1, *]}$ and $i \in S(\tree, \omega)$. 
\end{remark}
We finally show the following corollary of Theorem \ref{main}.  
\begin{corollary} \label{maincor}
    Let $(\{ f_{i} \}_{i \in I}, \tree)$ be a general IFS with the uniform contraction constant $c \in (0, 1)$ on a complete metric space $(X, \rho)$ and $z_{i} \in X$ be the unique fixed point of $f_{i} \ (i \in I)$. 
    Suppose that there exists $x \in X$ such that 
    \begin{equation*}
        \alpha := \limsup_{n \to \infty} \sqrt[n]{\max_{i \in I_{n}} \rho(x, z_{i}) } < \frac{1}{c}. 
    \end{equation*}
    Then, for each $r \in \{ r > 0 \ | \ c \leq r < 1, \alpha c < r \}$, we have the following properties: 
    \begin{enumerate}
        \item $\displaystyle \bigcup_{\omega^{\prime}_{[1, n]} \in \tree_{\omega}^{[1, n]}} f_{\omega^{\prime}_{[1, n]}} (\pi(\tree_{\omega \omega^{\prime}})) = \pi(\tree_{\omega})$ for each $\omega \in \tree^{[1, *]}$ \quad and  
        \item $(c/r)^{|\omega|+1} \cdot \rho_{H}( A, \pi(\tree_{\omega})) \ (\omega \in \tree^{[1, *]})$ is bounded for each $A \in \mathcal{K}(X)$. 
    \end{enumerate}
    In addition, the family of compact subsets $\{K_{\omega}\}_{\omega \in \tree}$ of $X$ with the above properties (i) and (ii) is unique. 
    Moreover, for all for all $A \in \mathcal{K}(X)$, $\omega \in \tree^{[1, *]}$, the sequence $\{ \bigcup_{\omega^{\prime}_{[1, n]} \in \tree^{[1, n]}_{\omega}} f_{\omega^{\prime}_{[1, n]}}(A)\}_{n \in \natural}$ of compact subsets converges to $\pi(\tree_{\omega})$ as $n$ tends to infinity exponentially fast with the rate $r$, in sense of the Hausdorff distance. 
\end{corollary}
\begin{proof}
    Let $r \in \{ r > 0 \ | \ c \leq r < 1, \alpha c < r \}$. 
    By the similar argument in the proof of Corollary \ref{expfastsufficond}, there exist $D^{\prime} > 0$ such that  
    \begin{align} \label{expfastconvcondition_Hausdorff}
        \left\{ \max_{i \in I_{n}} \rho(x^{\prime}, z_{i}) \right\} \cdot c^{n} \leq D^{\prime} \cdot r^{n}
    \end{align}
    for all $n \in \mathbb{N}$. 
    It follows that the condition in Theorem \ref{main} is satisfied and we obtain the property (i)  in Corollary \ref{maincor} for each $\omega \in \tree^{[1, *]}$, and by the inequalities (\ref{basic_inequality1_in_main_theorem}), (\ref{basic_inequality2_in_main_theorem}), (\ref{property2_lemma_main_theorem}) and (\ref{expfastconvcondition_Hausdorff}) it follows that
    for each $n \in \natural$, $\omega \in \tree^{[1, *]}$ and $A \in \mathcal{K}(X)$, 
    \begin{align} 
        &\rho_{H}\left( \bigcup_{\omega^{\prime}_{[1, n]} \in \tree_{\omega}^{[1, n]} } f_{\omega^{\prime}_{[1, n]}} ( \{ x \} ), \bigcup_{\omega^{\prime}_{[1, n]} \in \tree_{\omega}^{[1, n]} }f_{\omega^{\prime}_{[1, n]}} (A) \right) 
        &\leq c^{n} \cdot \sup_{a \in A} \rho(x, a), \label{basic_inequality1_in_main_cor}
    \end{align}
    \begin{align}
        \rho_{H}\left( \bigcup_{\omega^{\prime}_{[1, n]} \in \tree_{\omega}^{[1, n]} } f_{\omega^{\prime}_{[1, n]}} ( \{ x \} ), \pi(\tree_{\omega}) \right)
        &\leq (1+c) \cdot c^{-(|\omega|+1)} \sum_{k = |\omega|+n+1}^{\infty} D^{\prime} \cdot r^{k} \quad \text{and}
        \label{basic_inequality2_in_main_cor} 
    \end{align}
    \begin{align} 
        &\rho_{H}\left( \{x\}, \bigcup_{\omega^{\prime}_{1} \in \tree_{\omega}^{1}} f_{\omega^{\prime}_{1}}(\{x\})\right)
        \leq (1+c) \cdot \max_{i \in I_{|\omega|+1}} \rho\left( x, z_{i} \right)
        \leq (1+c) \cdot D^{\prime} \cdot \left( \frac{r}{c} \right)^{|\omega|+1}. 
        \label{property2_lemma_main_cor} 
    \end{align}
    
    We next show the properties (ii) in Corollary \ref{maincor}. 
    Let $A \in \mathcal{K}(X)$ and $\omega \in \tree^{[1, *]}$. 
    By the inequalities (\ref{property2_lemma_main_cor}) and (\ref{basic_inequality2_in_main_cor}) with $n = 1$, it follows that
    \begin{align}
        \rho_{H}( A, \pi(\tree_{\omega})) \notag
        &\leq \rho_{H}(A, \{x\}) + \rho_{H}\left( \{x\}, \bigcup_{\omega^{\prime}_{1} \in \tree_{\omega}^{1}} f_{\omega^{\prime}_{1}}(\{x\}) \right) + \rho_{H}\left( \bigcup_{\omega^{\prime}_{1} \in \tree_{\omega}^{1}} f_{\omega^{\prime}_{1}}(\{x\}), \pi(\tree_{\omega}) \right) \notag \\
        &\leq \sup_{a \in A} \rho(a, x) + (1+c) \cdot D^{\prime} \cdot \left( \frac{r}{c} \right)^{|\omega|+1} + (1+c) \cdot c^{-(|\omega|+1)} \sum_{k = |\omega|+2}^{\infty} D^{\prime} \cdot r^{k} \notag \\
        &\leq \sup_{a \in A} \rho(a, x) + (1+c) D^{\prime} \cdot c^{-(|\omega|+1)} \sum_{k = |\omega|+1}^{\infty} r^{k} \notag \\
        &\leq \max \left\{ \sup_{a \in A} \rho(a, x), \frac{1+c}{1-r} D^{\prime} \right\} \cdot \left( \frac{r}{c} \right)^{|\omega|+1}. \label{property2_rigorous_main_cor}
    \end{align}
    Thus, we have proved the properties (ii) in Corollary \ref{maincor}. 
    
    We now show that the uniqueness of the family of the compact sets with the properties (i) and (ii) in Corollary \ref{maincor}, let $\{K_{\omega}\}_{\omega \in \tree^{[1, *]}}$ be a family of the compact sets with the properties. 
    Also, Let $\omega \in \tree^{[1, *]}$ and $n \in \natural$. 
    Then, by the same argument in the proof of the uniqueness of in Theorem \ref{main}, we have
    \begin{align*}
        \rho_{H}(\pi(\tree_{\omega}), K_{\omega})
        \leq c^{n} \rho_{H}\left( \pi( \tree_{\omega\omega^{\prime}_{[1, n]}}), \{x\} \right) + c^{n} \rho_{H}( \{x\} , K_{\omega\omega^{\prime}_{[1, n]}} ). 
    \end{align*}
    Moreover, by the inequality (\ref{property2_rigorous_main_cor}) with $A := \{x\}$, we deduce that 
    \begin{align*}
        c^{n} \rho_{H} \left( \{x\}, \pi(\tree_{\omega\omega^{\prime}_{[1, n]}}) \right)
        \leq c^{n} \cdot \frac{1+c}{1-r} D^{\prime} \cdot \left( \frac{r}{c} \right)^{|\omega\omega^{\prime}_{[1, n]}|+1}
        = \frac{1+c}{1-r} D^{\prime} \cdot c^{-(|\omega|+1)} \cdot r^{|\omega|+n+1} \to 0 
    \end{align*}
    as $n$ tends to infinity, and by the same argument $c^{n} \rho_{H}( \{x\}, K_{\omega\omega^{\prime}_{[1, n]}}) \to 0$ as $n$ tends to infinity. 
    It follows that $\pi(\tree_{\omega}) = K_{\omega}$ for each $\omega \in \tree^{[1, *]}$. 
    Thus, we have proved the uniqueness of the family of the compact sets with the properties (i) and (ii) in Theorem \ref{main}.     
    
    We finally show that, for all $A \in \mathcal{K}(X)$, $\omega \in \tree^{[1, *]}$, $\{\cup_{\omega_{[1, n]} \in \tree^{n}} f_{\omega_{[1, n]}}(A)\}_{n \in \natural}$ converges to $\pi(\tree)$ as $n$ tends to infinity exponentially fast with the rate $r$, in sense of the Hausdorff distance. 
    To show this, let $A \in \mathcal{K}(X)$, $\omega \in \tree^{[1, *]}$ and $n \in \natural$. 
    Then, by the inequalities (\ref{basic_inequality1_in_main_cor}) and  (\ref{basic_inequality2_in_main_cor}), we deduce that 
    \begin{align*}
        &\rho_{H}\left( \bigcup_{\omega^{\prime}_{[1, n]} \in \tree_{\omega}^{[1, n]} } f_{\omega^{\prime}_{[1, n]}} (A), \pi(\tree_{\omega}) \right) \\
        &\leq \rho_{H}\left( \bigcup_{\omega^{\prime}_{[1, n]} \in \tree_{\omega}^{[1, n]} } \! \! \! f_{\omega^{\prime}_{[1, n]}} (A), \bigcup_{\omega^{\prime}_{[1, n]} \in \tree_{\omega}^{[1, n]} } \! \! \! f_{\omega^{\prime}_{[1, n]}} (\{ x \}) \right) + \rho_{H}\left( \bigcup_{\omega^{\prime}_{[1, n]} \in \tree_{\omega}^{[1, n]} } \! \! \! f_{\omega^{\prime}_{[1, n]}} ( \{ x \} ), \pi(\tree_{\omega}) \right)\\
        &\leq c^{n} \cdot \sup_{a \in A} \rho(x, a) + (1+c) \cdot c^{-(|\omega|+1)} \sum_{k = |\omega|+n+1}^{\infty} D^{\prime} \cdot r^{k} \\
        &\leq \sup_{a \in A} \rho(x, a) \cdot r^{n} + \frac{1+c}{1-r} c^{-(|\omega|+1)} D^{\prime} r^{|\omega|+n+1}  
        \leq \max \left\{ \sup_{a \in A} \rho(x, a), \frac{1+c}{1-r} D^{\prime} \left( \frac{r}{c} \right)^{|\omega|+1} \right\} \cdot r^{n}.  
    \end{align*}
    Hence, we have proved our theorem. 
\end{proof}
\begin{remark}
    By the same argument in Remark \ref{independence_of_alpha}, if $\{ \max_{i \in I_{n}} \rho(x, z_{i}) \}_{n \in \natural}$ is unbounded ( if and only if $\{z_{i} \ | \ i \in \cup_{n \in \natural} I_{n} \ \} \subset X$ is unbounded), then the constant $\alpha \geq 0$ in Corollary \ref{maincor} does not depend on $x \in X$. 
    Note that $\{ \pi(\tree_{\omega}) \}_{\omega \in \tree^{[1, *]}}$ is not uniformly bounded in general even if the assumption in Corollary \ref{maincor} holds. 
    Indeed, if the tree is a single set, then the general IFSs is reduced to the non-autonomous iteration (see, Example \ref{unboundedex}). 
    
    On the other hand, by the same argument in Remark \ref{independence_of_alpha}, we also deduce that if $\{ \max_{i \in I_{n}} \rho(x, z_{i}) \}_{n \in \natural}$ is bounded ( if and only if $\{z_{i} \ | \ i \in I \ \} \subset X$ is bounded), then the constant $\alpha \geq 0$ in Corollary \ref{maincor} depend on $x \in X$. 
    However, by the same argument in Remark \ref{expfast_remark}, if $\{ \max_{i \in I_{n}} \rho(x, z_{i}) \}_{n \in \natural}$ is bounded, then the condition in Theorem \ref{main} is automatically satisfied and $\{ \pi(\tree_{\omega}) \}_{\omega \in \tree^{[1, *]}}$ is uniformly bounded by the property (ii) in Theorem \ref{main}. 
\end{remark}
%
%
%
%
%
%
%
%
%
%
%
%
%
%
%
%
%
%
%
%
%
%
%
%
%
%
%
%
%
%
%
%
%
%
%
%
%
%
%
%
%
%
%
%
%
%
%
%
%
%
%
%
\section{An example of general IFSs}
In this section, we consider an example of general IFSs and the limit sets. 
Indeed, We first give an example of general IFSs which has a connection to the theory of continued fractions and we later give a proposition to indicate the importance of the example of the limit sets. 
Note that while the theory of continued fractions is often discussed in the theory of autonomous IFSs (see, \cite{MU}, \cite{MU2}), it is not often discussed in the theory of generalized IFSs (you can find a recent paper \cite{NT} in the setting for non-autonomous IFSs). 
In addition, while we already obtain the existence of the limit set generated by the IFS in the example by applying results in the third line, this example is not much paid attention to the limit set since it does not satisfy the central condition (the $V$-variability). 
Therefore, it is important to describe an example of the limit set for general IFSs even if the space $X$ is bounded. 

We now give the setting of the example of general IFSs. 
Let $I := \natural$ and $X := \{ z \in \complex \ | \ |z-1/2| \leq 1/2 \}$ where $| \cdot |$ is the Euclidean metric on $\complex$. 
For each $b \in I$, $S_{I} := \{ \phi_{b} \colon X \rightarrow X \ |\ b \in I \} $ is called the IFS of regular continued fractions. 
Here, 
\[ \phi_{b}(z) := \frac{1}{z + b} \quad ( z \in X ). \]
Note that for all $b \in I$, $\phi_{b}(X) \subset X$. 
Indeed, let $Y := \{ z \in \complex |\ \Re z \geq 1 \}$ and let $f \colon \hat{\complex} \to \hat{\complex} $ be the M\"{o}bius transformation defined by $f(z) := 1/z$ \ ($z \in \hat{\complex}$). 
Since $f(0) = \infty$, $f(1) = 1$, $f(1/2+i/2) = 2/(1+i) = (1-i)$, we have $f(\partial X) = \partial Y \cup \{ \infty \} $ and Since f(1/2) = 2, we have $f(X) = Y \cup \{ \infty \}$. 
Therefore, $f \colon X \to Y \cup \{ \infty \}$ is a homeomorphism and 
we deduce that $\phi_{b} = f^{-1} \circ g_{b}$ and $\phi_{b}(X) \subset f^{-1}(Y) \subset X$, where $g_{b} \colon X \to Y$ be the map defined by $g_{b}(z) := z+b$ \ ($b \in I$). 
    
Moreover, $S_{I}$ is a family of contractive mappings on $X$ with uniform contraction constant $c := 4/5$. 
Indeed, note that 
\begin{align*}
    |z+b|^{2}	
    &= |x+b+iy|^{2}
	= (x+b)^{2}+y^{2} 
	= x^{2} + 2bx + b^{2} + y^{2}
	\geq |z|^{2} + b^{2}
	=\frac{5}{4}
\end{align*}
for each $z=x+iy \in X$ and for each $b \in I$. 
It follows that 
\begin{align*}
    |\phi_{b}(z)-\phi_{b}(z^{\prime})|	
    &= \left| \displaystyle \frac{1}{z+b} - \frac{1}{z^{\prime}+b} \right|
	= \frac{|z-z^{\prime}|}{|z+b||z^{\prime}+b|} 
	\leq \frac{4}{5}|z-z^{\prime}|
\end{align*}
for each $z, z^{\prime} \in X$. 
Therefore, $S_{I}$ satisfies the condition (ii) in Definition \ref{def_of_genral_NIFS}.

\begin{example}
    Let $\alpha > 1$ and 
    \begin{align*}
        \tree_{\alpha} := \{ \omega = \omega_{1}\omega_{2} \cdots \in I^{\natural} \ |  \ \omega_{1} + \cdots + \omega_{n} < n \alpha \ \text{for each} \ n \in \natural \ \}.    
    \end{align*}
    We show that $(S_{I}, \tree_{\alpha})$ is a general IFS. 
    To show this, note that $ 1 \ 1 \ 1 \cdots \in I^{\natural}$ is a element of $\tree_{\alpha}$ since $\alpha > 1$. 
    Also, note that $\tree_{\alpha} \subset I^{\natural}$ is closed since we endow $I$ with the discrete topology and $I^{\natural}$ with the product topology, and $\Pi_{[1, n]} \colon I^{\natural} \to I^{n}$ is continuous on $I^{\natural}$ for each $n \in \natural$. 
    It remains to show that $\tree_{\alpha}$ is a tree with $I$. 
    It is easy to show that $\Pi_{1}(\tree_{\alpha})$ is finite. 
    Let $n \in \natural$ and $\omega \in \tree_{\alpha}$. 
    Then, 
    if $\omega_{n+1} \geq (n+1) \alpha - \sum_{l = 1}^{n} \omega_{l} \ (> 0)$, then we deduce that 
    \begin{align*}
        \omega_{1} + \cdots + \omega_{n} + \omega_{n+1} 
        \geq \sum_{l = 1}^{n} \omega_{l} + (n+1) \alpha - \sum_{l = 1}^{n} \omega_{l}
        = (n+1) \alpha.  
    \end{align*}
    We obtain that $\# ( S(\tree_{\alpha}, \omega_{[1, n]}) ) < \infty$ for each $n \in \natural$ and $\omega \in \tree_{\alpha}$, and $\tree_{\alpha}$ is a tree with $I$. 
    Thus, we have proved that $(S_{I}, \tree_{\alpha})$ is a general IFS. 
    
    Note that, by Definitions \ref{definition_of_projection_map} and \ref{Definition_of_limit_set}, the limit set for ($S_{I}$, $\tree_{\alpha}$) is the following non-empty and compact subset: 
    \begin{align*}
        \pi(\tree_{\alpha}) 
        &= \{ x_{1} \in X \ | \ \omega \in \tree_{\alpha},  \{x_{m}\}_{m \in \natural} \ \text{is recursively compatible for} \ \{f_{\omega_{n}}\}_{n \in \natural} \ \} \\
        &= \left\{ [0; \omega_{1}, \omega_{2}, \ldots] \in X \ | \ \omega_{1} + \omega_{2} \cdots + \omega_{n} \leq n \alpha \ \text{for each} \ n \in \natural \ \right\}, 
    \end{align*}
    where $[0; \omega_{1}, \omega_{2}, \ldots]$ is the continued fraction defined by
    \begin{align*}
        [0; \omega_{1}, \omega_{2}, \ldots] := \cfrac{1}{\omega_{1} + \cfrac{1}{\omega_{2} + \cdots}} \quad (\omega := \omega_{1}\omega_{2} \cdots \in I^{\natural} ). 
    \end{align*}
    Here, we use the general theory of continued fractions (for example, see \cite{Kh}). 
\end{example}
Now, we give a proposition to indicate the importance of the limit sets in the above example by using the results in \cite{CV}. 
For $A \subset \complex$, we denote by $\dim_{\mathcal{H}}A$ the Hausdorff dimension of $A$. 
\begin{proposition}\label{continued_fraction__elementray_lemma}
    Let $(S_{I}, \tree_{\alpha})$ 
    be general IFSs defined above. 
    We set
    \begin{align*}
        X_{\alpha} &:= \left\{ [0; \omega_{1}, \omega_{2}, \ldots] \in X \ | \ \omega = \omega_{1}\omega_{2} \cdots \in I^{\natural}, \ \limsup_{n \to \infty} \frac{1}{n} \sum_{i = 1}^{n} \omega_{i} < \alpha \ \right\}. 
    \end{align*}
    Then, we have 
    $\dim_{\mathcal{H}} X_{\alpha} = \dim_{\mathcal{H}} \pi(\tree_{\alpha} )$ and $\dim_{\mathcal{H}} X_{\alpha} (= \dim_{\mathcal{H}} \pi(\tree_{\alpha} ) )$ converges to $1$ as $\alpha$ tends to infinity. 
\end{proposition}
\begin{proof}
    Let $\alpha > 1$. 
    We first show that $\dim_{\mathcal{H}} X_{\alpha} \leq \dim_{\mathcal{H}} \pi(\tree_{\alpha} )$
    For each $N \geq 2$, we set 
    \begin{align*}
        \mathcal{I}_{N} 
        &:= \left\{ \ \tau \in I^{N-1} \ | \ 
        \begin{array}{c}
            \text{there exists} \ \omega = \omega_{1}\omega_{2} \cdots \in I^{\natural} \ \text{s.t.}  \\
             \omega_{[1, N-1]} = \tau, \ \sum_{i = 1}^{n} \omega_{i} < n \alpha \ \text{for all} \ n \geq N
        \end{array}
         \ \right\} \\
         &= \left\{ \ \tau \in I^{N-1} \ | \ 
         \begin{array}{c}
             \tau = \tau_{1} \cdots \tau_{N-1}, \  
             \sum_{i = 1}^{N-1} \tau_{i} + (n - N + 1) < n \alpha \ \text{for all} \ n \geq N
        \end{array}
         \ \right\} \\
         &= \textstyle \left\{ \ \tau \in I^{N-1} \ | \ \tau = \tau_{1} \cdots \tau_{N-1}, \ \sum_{i = 1}^{N-1} \tau_{i} + 1 < N \alpha \ \right\}. 
    \end{align*}
    Note that $\limsup_{n \to \infty} ( \sum_{i = 1}^{n} \omega_{i} )/n < \alpha $ if and only if there exists $N \in \natural$ with $N \geq 2$ such that $\sum_{i = 1}^{n} \omega_{i} < n \alpha$ for all $n \in \natural$ with $n \geq N$, and by direct calculations we have
    \begin{align}
        &\left\{ [0; \omega_{1}, \omega_{2}, \ldots] \in X \ | \ 
        \begin{array}{c}
            \omega = \omega_{1}\omega_{2} \cdots \in I^{\natural}, \omega_{[1, N-1]} = \tau \ \text{and} \ \\
            \sum_{i = 1}^{n} \omega_{i} < n \alpha \ \ \text{for all} \ n \geq N 
        \end{array}
        \ \right\} \notag \\
        & = \left\{ \phi_{\omega_{[1, N-1]}}( [0; \omega_{N}, \omega_{N+1}, \ldots] ) \in X \ | \ 
        \begin{array}{c}
            \omega = \omega_{1}\omega_{2} \cdots \in I^{\natural}, \omega_{[1, N-1]} = \tau \ \  \text{and} \\
            \text{ for all } \ n \geq N, \ \sum_{i = 1}^{N-1} \tau_{i} + \sum_{i = N}^{n} \omega_{i} < n \alpha
        \end{array}
        \ \right\} \notag \\
        & = \phi_{\tau} \left( \left\{ [0; \omega_{1}, \omega_{2}, \ldots] \in X \ | \ 
        \begin{array}{c}
            \omega = \omega_{1}\omega_{2} \cdots \in I^{\natural} \ \ \text{and \ for all} \ m \geq 1, \\
            \sum_{i = 1}^{N-1} \tau_{i} + \sum_{i = 1}^{m} \omega_{i} < m \alpha + (N-1) \alpha 
        \end{array}
        \ \right\} \right) \label{conti_frac_calc} \notag
    \end{align}
    for each $N \geq 2$ and $\tau = \tau_{1} \cdots \tau_{N-1} \in \mathcal{I}_{N}$. 
    By the above arguments, we deduce that 
    \begin{align*}
        X_{\alpha}
        &= \bigcup_{N = 2}^{\infty} \left\{ [0; \omega_{1}, \omega_{2}, \ldots] \in X \ | \ \omega = \omega_{1}\omega_{2} \cdots \in I^{\natural}, \ \sum_{i = 1}^{n} \omega_{i} < n \alpha \ \text{for all} \ n \geq N \  \right\} \\
        &= \bigcup_{N = 2}^{\infty} \bigcup_{ \tau \in \mathcal{I}_{N}} \left\{ [0; \omega_{1}, \omega_{2}, \ldots] \in X \ | \ 
        \begin{array}{c}
            \omega = \omega_{1}\omega_{2} \cdots \in I^{\natural}, \omega_{[1, N-1]} = \tau \\
            \sum_{i = 1}^{n} \omega_{i} < n \alpha \ \text{for all} \ n \geq N 
        \end{array}
        \ \right\} \\
        &= \bigcup_{N = 2}^{\infty} \bigcup_{ \tau \in \mathcal{I}_{N}} \phi_{\tau} \left( \left\{ [0; \omega_{1}, \omega_{2}, \ldots] \in X \ | \ 
        \begin{array}{c}
            \omega = \omega_{1}\omega_{2} \cdots \in I^{\natural} \ \  \text{and \ for all} \ m \geq 1, \\
            \sum_{i = 1}^{N-1} \tau_{i} + \sum_{i = 1}^{m} \omega_{i} < m \alpha + (N-1) \alpha 
        \end{array}
        \ \right\} \right) \\
        &\subset \bigcup_{N = 2}^{\infty} \bigcup_{ \tau \in \mathcal{I}_{N}} \phi_{\tau} \left( \left\{ [0; \omega_{1}, \omega_{2}, \ldots] \in X \ | \ 
        \begin{array}{c}
            \omega = \omega_{1}\omega_{2} \cdots \in I^{\natural} \ \  \text{and \ for all} \ m \geq 1, \\
            \sum_{i = 1}^{N-1} 1 + \sum_{i = 1}^{m} \omega_{i} < m \alpha + (N-1) \alpha 
        \end{array}
        \ \right\} \right). 
    \end{align*}
    Therefore, since $\phi_{\tau}$ is bi-Lipschitz on $X$ for each $N \geq 2$ and $\tau \in \mathcal{I}_{N}$ (for example, see \cite{MU}, \cite{IOS}), we deduce that
    \begin{align*}
        &\dim_{\mathcal{H}}X_{\alpha} \\
        &\leq \sup_{N \geq 2} \sup_{\tau \in \mathcal{I}_{N}} \dim_{\mathcal{H}} \phi_{\tau} \left( \left\{ [0; \omega_{1}, \omega_{2}, \ldots] \in X \ | \ 
        \begin{array}{c}
            \omega = \omega_{1}\omega_{2} \cdots \in I^{\natural} \ \  \text{and \ for all} \ m \geq 1, \\
            \sum_{i = 1}^{N-1} 1 + \sum_{i = 1}^{m} \omega_{i} < m \alpha + (N-1) \alpha 
        \end{array}
        \ \right\} \right) \\
        &= \sup_{N \geq 2} \sup_{\tau \in \mathcal{I}_{N}} \dim_{\mathcal{H}} 
        \left( \left\{ [0; \omega_{1}, \omega_{2}, \ldots] \in X \ | \ 
        \begin{array}{c}
            \omega = \omega_{1}\omega_{2} \cdots \in I^{\natural} \ \  \text{and \ for all} \ m \geq 1, \\
            \sum_{i = 1}^{N-1} 1 + \sum_{i = 1}^{m} \omega_{i} < m \alpha + (N-1) \alpha 
        \end{array}
        \ \right\} \right) \\
        &= \sup_{N \geq 2} \dim_{\mathcal{H}} \left( \left\{ [0; \omega_{1}, \omega_{2}, \ldots] \in X \ | \ 
        \begin{array}{c}
            \omega = \omega_{1}\omega_{2} \cdots \in I^{\natural}, \  \omega_{[1, N-1]} = \mathbb{I}, \\
            \sum_{i = 1}^{n} \omega_{i} < n \alpha \text{\ for all} \ n \geq N
        \end{array}
        \ \right\} \right) \\
        &\leq \dim_{\mathcal{H}} \left( \left\{ [0; \omega_{1}, \omega_{2}, \ldots] \in X \ | \ 
        \begin{array}{c}
            \omega = \omega_{1}\omega_{2} \cdots \in I^{\natural}, \\
            \sum_{i = 1}^{n} \omega_{i} < n \alpha \text{\ for all} \ n \geq 1
        \end{array}
        \ \right\} \right)
        = \dim_{\mathcal{H}} \pi(\tree_{\alpha}), 
    \end{align*}
    where $\mathbb{I} = \underbrace{1 \cdots 1}_{N-1} \in \mathcal{I}_{N}$. 
    To show that $\dim_{\mathcal{H}} X_{\alpha} \geq \dim_{\mathcal{H}} \pi(\tree_{\alpha} )$, Let $\epsilon \in (0, \alpha -1)$.  
    Since
    \begin{align*}
        \pi(\tree_{\alpha - \epsilon}) 
        &\subset \{ [0; \omega_{1}, \omega_{2}, \ldots ] \in X \ | \ \omega = \omega_{1}\omega_{2} \cdots \in I^{\natural}, \ \limsup_{n \to \infty} \frac{1}{n} \sum_{i = 1}^{n} \omega_{i} \leq \alpha - \epsilon \ \} \\
        &\subset \{ [0; \omega_{1}, \omega_{2}, \ldots ] \in X \ | \ \omega = \omega_{1}\omega_{2} \cdots \in I^{\natural}, \ \limsup_{n \to \infty} \frac{1}{n} \sum_{i = 1}^{n} \omega_{i} < \alpha \ \}
        = X_{\alpha}, 
    \end{align*}
    we deduce that $\dim_{\mathcal{H}} \pi(\tree_{\alpha - \epsilon}) \leq \dim_{\mathcal{H}} X_{\alpha}$. 
    Now, by the continuity of the dimension function $\alpha \mapsto \dim_{\mathcal{H}} \pi(\tree_{\alpha})$ (Theorem 1 in \cite{CV}), it follows that $\dim_{\mathcal{H}} \pi(\tree_{\alpha} ) \leq \dim_{\mathcal{H}} X_{\alpha}$. 
    Also, by Theorem 1 in \cite{CV}, we have proved the rest of the claims.    
    Hence, we have proved our proposition. 
\end{proof}
\section*{Acknowledgment} 
The author would like to thank Hiroki Sumi, Yuto Nakajima and Mitsuhiro Shishikura for giving me helpful comments in Sections 3 and 4. 
The author also would like to thank Shunsuke Usuki for giving me helpful comments in Section 5. 
The author is supported by JST CREST Grant Number JPMJCR1913.  


\begin{thebibliography}{99}
\bibitem{A} Atnip, J. (2017). Non-autonomous conformal graph directed Markov systems. arXiv preprint arXiv:1706.09978. 
\bibitem{BG} Bandt, C., \& Graf, S. (1992). Self-similar sets 7. A characterization of self-similar fractals with positive Hausdorff measure. Proceedings of the American Mathematical Society, 995-1001.  
\bibitem{BH} Barlow, M. T., \& Hambly, B. M. (1997, January). Transition density estimates for Brownian motion on scale irregular Sierpinski gaskets. In Annales de l'Institut Henri Poincare (B) Probability and Statistics (Vol. \textbf{33}, No. 5, pp. 531-557). No longer published by Elsevier. 
\bibitem{B} Barnsley, M. F. (2014). \textit{Fractals everywhere}. Academic press. 
\bibitem{BHS1} Barnsley, M. F., Hutchinson, J. E., \& Stenflo, Ö. (2008). V-variable fractals: fractals with partial self similarity. Advances in Mathematics, \textbf{218}(6), 2051-2088. 
\bibitem{BHS2} Barnsley, M., Hutchinson, J. E., \& Stenflo, Ö. (2012). V-variable fractals: dimension results.
\bibitem{CV} Cesaratto, E., \& Vallée, B. (2006). Hausdorff dimension of real numbers with bounded digit averages. Acta Arithmetica, 115-162. 
\bibitem{DLM} Dyn, N., Levin, D., \& Massopust, P. (2020). Attractors of trees of maps and of sequences of maps between spaces with applications to subdivision. Journal of Fixed Point Theory and Applications, \textbf{22}(1), 14. 
\bibitem{F} Falconer, K. (2004). \textit{Fractal geometry: mathematical foundations and applications}. John Wiley \& Sons. 
\bibitem{GM} Gu, Y., \& Miao, J. J. (2022). Dimensions of a class of self-affine Moran sets. Journal of Mathematical Analysis and Applications, \textbf{513}(1), 126210.
\bibitem{HZ} Holland, M., \& Zhang, Y. (2013). Dimension results for inhomogeneous Moran set constructions. Dynamical Systems, \textbf{28}(2), 222-250.
\bibitem{HRWW} Hua, S.,  Rao, H.,  Wen, Z.,  Wu, J., On the structures and dimensions of Moran sets. (English summary) Sci. China Ser. A \textbf{43}(8), 836-852.
\bibitem{H} Hutchinson, J. E. (1981). Fractals and self similarity. Indiana University Mathematics Journal, \textbf{30}(5), 713-747. 
\bibitem{I} Inui, K. (2020) Study of the fractals generated by contractive mappings and their dimensions, Ph.D. Thesis,
Kyoto University, available at https://repository.kulib.kyoto-u.ac.jp/dspace/handle/2433/253370
\bibitem{IOS} Inui, K., Okada, H., \& Sumi, H. (2020). The Hausdorff dimension function of the family of conformal iterated function systems of generalized complex continued fractions. Discrete Contin. Dyn. Syst. \textbf{40}(2), 753-766. 
\bibitem{Kh} Khinchin, A. Y., \& Teichmann, T. (1964). \textit{Continued fractions}. Physics Today, \textbf{17}(11), 70. 
\bibitem{Ki} Kigami, J. (2001). \textit{Analysis on fractals} (No. \textbf{143}). Cambridge University Press. 
\bibitem{KTMV} Kunze, H., La Torre, D., Mendivil, F., \& Vrscay, E. R. (2011). \textit{Fractal-based methods in analysis}. Springer Science \& Business Media. 
\bibitem{LDV} Levin, D., Dyn, N., \& Puthan Veedu, V. (2019). Non-stationary versions of fixed-point theory, with applications to fractals and subdivision. Journal of Fixed Point Theory and Applications, \textbf{21}, 1-25.
\bibitem{LLMX} Li, W., Li, W., Miao, J., \& Xi, L. (2016). Assouad dimensions of Moran sets and Cantor-like sets. Frontiers of Mathematics in China, \textbf{11}, 705-722.
\bibitem{Mas} Massopust, P. (2019). Non-stationary fractal interpolation. Mathematics, \textbf{7}(8), 666. 
\bibitem{MU} Mauldin, R. D., \& Urbański, M. (1996). Dimensions and measures in infinite iterated function systems. Proceedings of the London Mathematical Society, \text{3}(1), 105-154.
\bibitem{MU2} Mauldin, R., \& Urbański, M. (1999). Conformal iterated function systems with applications to the geometry of continued fractions. Transactions of the American Mathematical Society, \textbf{351}(12), 4995-5025.
\bibitem{MU3} Mauldin, R. D., \& Urbanski, M. (2003). \textit{Graph directed Markov systems: geometry and dynamics of limit sets} (Vol. \textbf{148}). Cambridge University Press.
\bibitem{Mo} Moran, P. A. (1946). Additive functions of intervals and Hausdorff measure. Proc. Cambridge Philos. Soc. \textbf{42}, 15-23. 
\bibitem{N} Nakajima, Y. (2022). Dimensions of slices through the Sierpiński gasket. Journal of Difference Equations and Applications, \textbf{28}(3), 429-456. 
\bibitem{NT} Nakajima, Y., \& Takahasi, H. (2022). Hausdorff dimension of sets with restricted, slowly growing partial quotients in the semi-regular continued fraction. arXiv preprint arXiv:2209.08318.
\bibitem{ReU} Rempe-Gillen, L., \& Urbański, M. (2016). Non-autonomous conformal iterated function systems and Moran-set constructions. Transactions of the American Mathematical Society, \textbf{368}(3), 1979-2017. 
\bibitem{Sce} Scealy, R. (2009). V-variable fractals and interpolation, Australian National University, Ph.D thesis.  
\bibitem{Sch} Schief, A. (1994). Separation properties for self-similar sets. Proceedings of the American Mathematical Society, \textbf{122}(1), 111-115. 
\end{thebibliography}
\end{document}